\documentclass[12pt,reqno]{amsart}
  
\usepackage{latexsym}
\usepackage{amssymb, tikz}

\usepackage{color}

\usepackage{amsthm}
\theoremstyle{plain}
\usepackage[title]{appendix}

\newtheorem*{theorem*}{Theorem}

\renewcommand{\epsilon}{\varepsilon}

\newcommand{\N}{{\mathbb N}}
\newcommand{\R}{{\mathbb R}}

\numberwithin{equation}{subsection}
\newtheorem{theo}[equation]{{\sc Theorem}}

\newtheorem{question}[equation]{{\sc Question}}

\newtheorem{conj}[equation]{{\sc Conjecture}}

\newtheorem{lem}[equation]{{\sc Lemma}}

\newtheorem{prop}[equation]{{\sc Proposition}}

\theoremstyle{definition}
\newtheorem{defn}[equation]{{\sc Definition}}

\theoremstyle{remark}
\newtheorem{rem}[equation]{Remark}
\theoremstyle{assumption}

\newtheorem{goal}[equation]{Goal}

\newtheorem{strategy}[equation]{\color{blue}{Strategy}}

\title[Distribution of tangents]{Tangent nodal sets for random spherical harmonics}
\author{Suresh Eswarathasan}
\address{(Permanent address) School of Mathematics, Cardiff University, Senghennyd Road, Cardiff, Wales CF24 4AG United Kingdom}
\email{eswarathasans@cardiff.ac.uk}
\address{(Current address) Department of Mathematics and Statistics, McGill University, Burnside Hall Rm 1208, 805 Sherbrooke Street West, Montr\'eal, Qu\'ebec H3A 0B9 Canada}
\email{suresh.eswarathasan@mcgill.ca}
\date{August 30, 2018}

\begin{document}

\begin{abstract}
In this note, we consider a fixed vector field $V$ on $S^2$ and study the distribution of points which lie on the nodal set (of a random spherical harmonic) where $V$ is also tangent.  We show that the expected value of the corresponding counting function is asymptotic to the eigenvalue with a leading coefficient that is independent of the vector field $V$.  This demonstrates, in some form, a universality for vector fields up to lower order terms.
\end{abstract}

\maketitle

\section{Introduction}
\subsection{Nodal Sets}
In the 18th century Ernest Chladni first described nodal sets during his study of modes of vibration on a rigid surface: the observed nodal pattern corresponds to the sets that remain stationary during vibrations.  To some extent, the manner in which these patterns develop as the frequency of the mode becomes larger has remained an enigma to mathematicians. 

Consider a compact boundaryless Riemannian manifold $(M,g)$ and the spectrum of the corresponding Laplace-Beltrami operator $-\Delta_g$ which we order as $\lambda_0^2 = 0 < \lambda_1^2 \leq \lambda_2^2 \leq \dots$ and tends to infinity.  Let us denote a corresponding $L^2$ orthonormal basis by $\{\varphi_j\}_j$ and let $Z(\varphi_j) = \{ x \in M: \varphi_j(x)=0 \}$ be its nodal set. Courant demonstrated in the 1920s that the number of connected components of $M\backslash Z(\varphi_j)$, commonly referred to as \textit{nodal domains}, is bounded above by a uniform constant times $j$.  The study of the nodal set and nodal domains of eigenfunctions under various assumptions is a well developed area of research and has demonstrated a number of connections to other areas of mathematics and mathematical physics; for more information, see the ICM article of Nazarov-Sodin \cite{NS10}.

In particular, much more can be said in settings which exhibit some rigid structure such as the torus or the sphere. In this note, we will consider an aspect of nodal sets (motivated by a question pertaining to nodal domains) for certain kinds of eigenfunctions on the 2-sphere $S^2$.

\subsection{Spherical Harmonics}
Consider the 2-dimensional sphere $S^2$, its positive Laplace-Beltrami operator $-\Delta_{g}$ where $g$ is the round metric, and normalized volume measure $dV_g$. Consider the eigenfunction equation
\begin{equation*}
-\Delta_g \varphi_{l} = \lambda_l^2 \varphi_{l},
\end{equation*}
where $l \in \N$.  Let $E_l$ denote the eigenspace for the eigenvalue $\lambda_l^2$.  We note that the eigenvalues $\lambda_l$ on the sphere are explicit and have large multiplicities, with the formulas for them being
\begin{equation*}
\lambda_l = l(l+1) \text{ and } \mathcal{N}_l= \dim E_{\lambda_l} = 2l+1.
\end{equation*}
Given $l$, we fix an $L^2$-orthonormal basis for $E_l$ which we denote by $\{ \varphi_{l,k} \}_{k=1}^{\mathcal{N}_l}$ and results in the identification $\R^{\mathcal{N}_l} \simeq E_l$.  In particular, using standard spherical coordinates on $S^2 \subset \R^3$, we set $\varphi_{l,k}:= e^{i k \theta} P_l^m(\cos \phi)$ where $P_l^m$ is the associated Legendre polynomial of degree $(l,m)$; this basis is commonly known as \textit{ultraspherical}.  For further reading on spherical harmonics we refer the reader to \cite{AAR99}.

\subsection{Random Model}
Next, we consider random eigenfunctions, that is, functions of the form
\begin{equation*}
f_l(x) = \sqrt{\frac{4 \pi}{\mathcal{N}_l}} \sum_{k=1}^{\mathcal{N}_l} a_k \varphi_{l,k}(x)
\end{equation*}
where $a_k$ are Gaussian $N(0,1)$ i.i.d. random variables.  Thanks to our identification of $E_l$ with $\R^{\mathcal{N}_l}$, we can put a Gaussian measure $\nu$ on $E_n$ with the expression
\begin{equation*}
d\nu_l(f) = \exp(-\|\overrightarrow{a}\|^2/2) \frac{da_1 da_2 \cdots da_{\mathcal{N}_l}}{(2 \pi)^{\mathcal{N}_l/2}},
\end{equation*}
where $\overrightarrow{a} = (a_k)_{k=1}^{\mathcal{N}_l} \in \mathbb{R}^{\mathcal{N}_l}$.  Note that the measure $\nu$ does not depend on our chosen basis $\{\varphi_{l,k}\}_k$.  Moreover, we see that $\mathbb{E}_l\left[f_l^2 \right]=1$ as an immediate consequence of the addition theorem for spherical harmonics.

Finally, we can consider the product measure $\nu := \otimes_{l=1}^{\infty} \nu_l$ on the space 
\begin{equation} \label{e:main_prob_space}
\Omega := \oplus_{l=1}^{\infty} E_l
\end{equation}
which can be seen as a probability space of random sequences of spherical harmonics.  

\subsection{Tangent Nodal Sets}
The inspiration for this note is the interaction between tangent/normal spaces to nodal sets of random eigenfunctions and various geometric quantities of these submanifolds, a notion which was also considered by Gayet-Welschinger who give upper and lower bounds on the expected Betti numbers (a natural step to understand the singularities of a vector field) for elliptic pseudodifferential operators \cite{GW14}, Dang-Riv\`ere (who themselves were motivated by the work \cite{GW14}) who give asymptotics pertaining to the equidistribution of normal cocycles for Laplace eigenfunctions on general compact manifolds \cite{DR17}, and Rudnick-Wigman who consider fixed normal directions to nodal sets on the flat torus \cite{RW18}.  Note that on the flat torus, it is possible to consider fixed non-vanishing directions $V_0$ unlike on general manifolds (consider the hairy ball theorem on the sphere).   

We seek in some sense a more singular cocycle to that of Dang-Rivi\`ere in that we would like to understand the current which places weight on the intersection of the normal bundle and that of the natural section generated by a general $V$ along the nodal set.  This work is a step towards this but in a very specific setting.

Let $\mathcal{V}(S^2)$ be the set of smooth vector fields on $S^2$ and consider a fixed $V \in \mathcal{V}(S^2)$.  Given an eigenfunction $f_l \in C^{\infty}(S^2)$, we can consider the action of the vector field $V$ on $f_l$ through the following formula,
\begin{equation*}
Vf_l(x) = \langle \nabla f, V \rangle_{g(x)},
\end{equation*}
where $\langle , \rangle_{g(x)}$ is the inner product on $T_xS^2$ and $\nabla f_l$ is the gradient of $f$ with respect to the round metric $g$.  Hence $\langle \nabla f_l, V \rangle_{g(x)} = 0$ is equivalent to $V(x) \in T_x \{f_l = 0 \}$.  That is, $V(f_l) = 0$ implies that $V \in T_x \{ f_l=0 \}$.

For the set of regular points in the nodal set $Z(f_l)$ (which is $\nu$-almost surely true thanks to Bulinskaya's Lemma; see the standard text \cite{AW09}), we would like to understand the statistics of when $V \in T_x(Z(f_l))$ or more succinctly the so-called \textit{V-tangent nodal set} $Z_V(f_l) = \{ x\in S^2 : f_l(x)=0, Vf_l(x) = 0 \}$ for a random eigenfunction $f_l$.  We will later show that $Z_V(f_l)$ is almost surely finite as well in Section \ref{subsec:prep_prob}.  In particular, we would like to understand the large-$l$ behavior of the expected value of $Z_V(f_l)$.

\subsection{Main Result}

The main result of this note is the following:

\begin{theo} \label{t:main}
Let $V \in \mathcal{V}(S^2)$ be fixed and have finitely many zeroes.  Let $m$ be the maximal order of vanishing amongst all the zeroes.

We have the following asymptotic for the expected value:
\begin{equation*}
\mathbb{E}_l[ \# Z_V(f_l) ] =  \frac{\sqrt{2}}{4\pi^2} \, l^2 + \mathcal{O}_{V}(l^{1 + \frac{3m}{3m + 2}}),
\end{equation*}
which holds for all $l \geq l_0(V)$.  Furthermore, the remainder term has a bounded dependence on $V$ and its derivatives.

In the case that 0 is a regular value of $V$ and of order $m$, $\{V=0\}$ is a smooth curve of finite length and we get a similar asymptotic but with a remainder term of $\mathcal{O}_{V}(l^{1 + \frac{3m}{3m + 1}})$
\end{theo}

\begin{rem}
The slightly larger than normal error term is due to weak singularities of the first intensity which arise from the zeroes of $V$, which are unavoidable due to topological reasons.  See Section \ref{subsec:proof_main} and proceeding comments for a detailed explanation.
\end{rem}

\begin{rem}
This form of independence contrasts that of Rudnick-Wigman \cite{RW18} where they consider the number of points in $Z(f_l)$ with a fixed normal direction $\zeta$ on the flat torus of dimension $d$; Rudnick-Wigman obtain an exact expected value where the angle $\zeta$ appears explictly as well as an upperbound for sequences of certain deterministic eigenfunctions.
\end{rem}

\subsection{Current work in progress}
This calculation stems from current work in progress on nodal domains where we have the following:

\begin{goal} \label{g:NS}
Find an explicit relationship, or some kind of substantial lack of relation, between $\mathbb{E}_l[ \# Z_V(f_l) ] $ and $\mathbb{E}_l [ \mbox{number of connected components of } Z(f_l) ]$.  More specifically, we would like to relate the constant $\frac{\sqrt{2}}{4 \pi^2}$ to $C_{NS}$ where the latter is the Nazarov-Sodin constant.
\end{goal}

The works of Nazarov-Sodin \cite{NS1, NS2} determine an asymptotic law on the counting function for the number of connected components of the nodal set of a random spherical harmonic $f_l$.  Let us state this specific result for reference purposes:

\begin{theo} \cite{NS1}
There exists a positive constant $C_{NS}$ (depending only on the dimension of $S^2$ and not its geometry) such that with probability tending to 1 as $l \rightarrow \infty$, 
\begin{equation*}
\mbox{the number of connected components of } Z(f_l) = \beta \frac{\omega_2}{(2\pi)^2} l^2 + o(l^2),
\end{equation*}
where $\omega_2$ is the volume of the unit ball in 2 dimensions. 
\end{theo}  

In the process of establishing this result (with the latter work actually addressing the more general hypotheses of Gaussian processes), the authors provide many fundamental ideas and powerful techniques which have allowed for the study of the topology of the connected nodal components of functions coming from Gaussian ensembles; see the papers of Canzani-Sarnak \cite{CS18} and Sarnak-Wigman \cite{SW18} as well as the series of works by Gayet-Welschinger (see \cite{GW14} and the references therein).

\textbf{Acknowledgements}: S.E. would like to thank I. Wigman for numerous helpful discussions and Z. Rudnick for suggesting the collaboration with Wigman which eventually led to this note and upcoming work. 

\section{Geometric Preliminaries}

\subsection{Coordinates $\&$ bases} Consider $S^2 \subset \R^3$ and take $x \in S^2$.  Throughout this note, we consider spherical coordinates at $x$ given by
\begin{equation*}
(\sin \phi_x \cos \theta_x, \sin \phi_x \sin \theta_x, \cos \phi_x)
\end{equation*} 
where $\theta_x \in [0, 2 \pi)$ and $\phi_x \in (0, \pi)$.  Using this system of coordinates, our metric on the set $[0, 2 \pi) \times (0, \pi)$ becomes
\begin{equation*}
g(\theta_x, \phi_x) = \begin{pmatrix}
\sin^2 \phi_x & 0 \\
0 & 1
\end{pmatrix}.
\end{equation*}
Throughout our computation, we will frequently use the orthogonal (instead of an orthonormal) basis $\{\frac{\partial}{\partial \theta_x},  \frac{\partial}{\partial \phi_x} \}$ due to the coordinate singularity at $\phi_x = 0$.  

\begin{rem}
Although the subscripts of $x$, which are meant to signify our coordinate representation is attached to the $x$-variable, may seem tedious notation-wise, it will become useful when calculating entries of the covariance matrix as we must consider various derivatives in $x$ and $y$ of the spectral projector $P_l(d(x,y))$ before setting $x=y$.
\end{rem}

\subsection{Vector field $V$}
We now let $V$ be a smooth vector field on $S^2$.  We note that in our chosen spherical coordinates, $V = v_1(\theta_x, \phi_x) \frac{\partial}{\partial \theta_x} + v_2(\theta_x, \phi_x) \frac{\partial}{\partial \phi_x}$ and define $V^{\perp} :=  v_2(\theta_x, \phi_x) \frac{\partial}{\partial \theta_x} + (- \sin^2{\phi_x}) v_1(\theta, \phi_x) \frac{\partial}{\partial \phi_x}$.  We choose $V^{\perp}$ in this particular way so that the ordered set $\{V, V^{\perp} \}$ has positive orientation and is orthogonal.  In two dimensions, we have the following consequence: given $V$, any other choice of $\tilde{V^{\perp}}$ for which $\{V, \tilde{V^{\perp}} \}$ is orthogonal and positively orientated is just a (variable) rescaling of $V^{\perp}$ at each point of $S^2$.

In regards to the action of $V$ on a smooth function $f$ in local coordinates denoted by $x$, after slightly abused notation, we have that
\begin{equation*}
Vf_l(x) = \langle \nabla f_l, V \rangle_{g(x)} = \nabla f_l(x)^T \cdot g \, V =: d(f_l)_x(V) = \nabla_{\theta_x, \phi_x} f_l \cdot V
\end{equation*}
where $\nabla_{\theta_x, \phi_x} f_l$ is the Euclidean gradient in the coordinates $\theta_x, \phi_x$ and have expressed the metric inner product through the multiplication of matrices; notice that we have used the geometric definition of the gradient in order to relate the metric gradient to the Euclidean gradient.  Hence we have the local expression $\nabla f_l(x) = g^{-1} \nabla_{\theta_x, \phi_x} f_l(x)$. 

We conclude with the quick observation that $\|V^{\perp} \|_g(x) = \|V \|_g(x) \sin \phi_x = \|V \|_g(x) \sqrt{\det g(x)}$, a fact which we will use in our calculations in Section \ref{subsec:calc_cond_cov}.

\section{Calculating the Expectation}

\subsection{Preparing the probability space} \label{subsec:prep_prob}

We must first verify the conditions of Theorem 6.2 in Aza\"s-Wchebor \cite{AW09}, particularly points (iii) and (iv).  The following lemma succinctly addresses these two conditions:

\begin{lem}
For the centered Gaussian field $F_l = (f_l, Vf_l)$, we have the distribution of $F_l$ is non-degenerate and that 
\begin{eqnarray}
\nonumber && \mathbb{P}_l[ NonSing_{V,l}^{\complement} := \{ \omega \in E_l: \exists x \in S^2 - \{ V=0 \} \mbox{ such that } f_l(x)=0, Vf_l(x)=0, \mbox{ and } \\
&& (DF_l)_x (x)\mbox{ is not invertible}\} ] = 0.
\end{eqnarray}   
\end{lem}
\begin{proof}
As calculated below, we find that the determinant of the covariance matrix for $F_l$ equals $\sqrt{2}\pi \|V\|_g \sqrt{l^2 +1}$.  The verification of the measure 0 property is slightly more subtle.  

It suffices to show that the differential over $S^2 \times \R^{2l+1}$ of $G=(f_l, Vf_l, \nabla f_l, \nabla Vf_l)$ has rank 4, for then the implicit function theorem would imply that $G^{-1}(\overrightarrow{0})$ is smooth and of codimension $4$ in $S^2 \times \R^{2l+1}$.  By dimension considerations, this implies that $\pi_2$, the projection onto $\R^{2l+1}$, applied to $G^{-1}(\overrightarrow{0})$ has dimension $\leq 2l-1$ and therefore the set $\{ \omega \in E_l: \exists x \in S^2 \mbox{ such that } f_l(x)=0, Vf_l(x)=0, \mbox{ and } (DF_l)_x (x)\mbox{ is degenerate} \}$ has Gaussian measure 0. 

Notice however that if we compute the differential over $\R^{2l+1}$ of the slightly modified Gaussian field $\tilde{G}=(f_l, Vf_l, \nabla f_l)$, which in coordinates gives a $4 \times (2l+1)$ sub-matrix of that of the full differential $DG$, and show this has rank 4  for $l$ large enough, then this would show our originally defined singular set has dimension bounded above by $2l-1$.  

Let $[D\tilde{G}]$ be the coordinate representation of our desired differential.  Then if $[D\tilde{G}][D\tilde{G}]^t$ is invertible on $\R^4$ then $[D\tilde{G}]$ has rank 4.  An explicit calculation using standard spherical coordinates gives us that
\begin{equation*}
[D\tilde{G}][D\tilde{G}]^t = 
\begin{pmatrix}
1 & 0 & 0 & 0 \\
0& \|V\|^2\frac{l^2+l}{2} & v_1 \sin^2 \phi_x\frac{l^2+l}{2} & v_2\frac{l^2+l}{2}\\
0 &   v_1 \sin^2 \phi_x\frac{l^2+l}{2}  & \frac{l^2+l}{2} & 0 \\
0 & v_2  \frac{l^2+l}{2} & 0 & \frac{l^2 + l}{2}.
\end{pmatrix}
\end{equation*}
The determinant of this matrix equals $\frac{(l^2 + l)^3}2 \left( \frac{\|V\|^2}{4} - \frac{v_1^4}{4}\sin \phi_x^4 - \frac{v_2^2}{4} \right)$ which can only be zero if $\|V\|^2 = v_1^2 \sin_x^4 \phi + v_2^2$.  Considering the round metric on $S^2$, this is possible only if $V=0$ at $x$ as $\sin_x^2 \phi \leq 1$.

\end{proof}

Hence, we can now identify $\Omega$ with $\oplus_l NonSing_{V,l}$. It is important to notice that we have not yet made any assumptions on the structure of the vanishing set $\{V(x) = 0\}$.  This will only play a role in the penultimate step of the proof of Theorem \ref{t:main} in Section \ref{subsec:proof_main}.

\subsection{Preparing the orthogonal determinant}
We will employ the Kac-Rice formula (see \cite{AW09} Chapter 6) and compute the quantity
\begin{equation*}
\int_{S^2} \Phi_l(0,0) \, \mathbb{E}\left[ | \det^{\perp} DF_l(x) | \big| f_l(x) = 0, V f_l(x) = 0 \right] \, dV_g(x)
\end{equation*} 
where $F_l = (f_l, Vf_l)$ and $\Phi_l$ is the Gaussian probability density function of $F_l$.  The quantity $\det^{\perp} DF_l$ is the orthogonal determinant, which is defined as the determinant of the map $DF_l^* DF_l$.  

It follows that evaluating this determinant at the orthogonal basis $\{V, V^{\perp}\}$ (i.e. expressing our coordinate basis vector fields with respect to the proposed orthogonal basis), at least away from the zeroes of $V$, with respect to spherical coordinates after using formulas for Legendre polynomials gives us
\begin{eqnarray}
&& \det^{\perp} DF_l = \left( \frac{a^{\theta_x}_{V}}{\|V\|^2_g}Vf_l + \frac{a^{\theta_x}_{V^{\perp}}}{\|V^{\perp}\|^2_g}V^{\perp}f_l \right) \left( \frac{a_{V}^{\phi}}{\|V\|^2_g} V Vf_l + \frac{a_{V^{\perp}}^{\phi}}{\|V^{\perp}\|^2_g} V^{\perp}Vf_l \right) \\
&& - \left( \frac{a^{\phi}_{V}}{\|V\|^2_g}Vf_l + \frac{a^{\phi}_{V^{\perp}}}{\|V^{\perp}\|^2_g}V^{\perp}f_l \right) \left(\frac{ a_{V}^{\theta_x}}{\|V\|^2_g} V Vf_l + \frac{a_{V^{\perp}}^{\theta_x}}{\|V^{\perp}\|^2_g} V^{\perp}Vf_l \right).
\end{eqnarray}  
Here, $a_{V}^{\theta_x} : = \langle \frac{\partial}{\partial \theta_x}, V \rangle_g = (\sin^2 \phi) v_1$, $a_{V}^{\phi} : = \langle \frac{\partial}{\partial \phi}, V \rangle_g = v_2$, $a_{V^{\perp}}^{\theta_x} : = \langle \frac{\partial}{\partial \theta_x}, V^{\perp} \rangle_g = (\sin^2 \phi) v_2$,  and $a_{V^{\perp}}^{\phi} : = \langle \frac{\partial}{\partial \phi}, V^{\perp} \rangle_g = (-\sin^2 \phi) v_1$; these coefficients follow immediately from our expression of $V$ and $V^{\perp}$ with respect to our coordinates.

Using the conditioning, the absolute value of the orthogonal determinant reduces to 
\begin{eqnarray} \label{e:orthog_det}
\nonumber && \frac{1}{\|V\|_g^2 \|V^{\perp} \|_g^2} \left| \left( a_{V^{\perp}}^{\theta_x}a_{V}^{\phi} - a_{V^{\perp}}^{\phi} a_{V}^{\theta_x} \right) (V^{\perp}f_l)(VV f_l) \right| \\
\nonumber && =  \frac{1}{\|V\|_g^2 \|V^{\perp} \|_g^2} (\sin^2 \phi) ( v_2^2 + \sin^2 \phi v_1^2) \left|  (V^{\perp}f_l)(VV f_l) \right| \\
\nonumber && =  \frac{1}{\|V\|_g^2 \|V^{\perp} \|_g^2} \cdot \| V \|_g^2 \det(g) \cdot \left| (V^{\perp}f_l)(VV f_l) \right| \\
\nonumber && = \frac{1}{ \|V^{\perp} \|_g^2}\det(g) \cdot \left| (V^{\perp}f_l)(VV f_l) \right| \\
&& = \frac{1}{ \|V \|_g^2} \cdot \left| (V^{\perp}f_l)(VV f_l) \right| 
\end{eqnarray}
where $g$ is the round metric in coordinates. Hence, we take the Gaussian field $(f_l, Vf_l, V^{\perp}f_l, VVf_l)$ and compute the corresponding conditional covariance matrix.  That is, we compute the covariance for the field $X_2 := (V^{\perp}f_l, VVf_l)$ conditioned on the event that $X_1 := (f_l, Vf_l) = 0.$ 

\subsection{Entries of the Covariance Matrix}
We list the coefficients of the full covariance matrix for the field $(f_l, Vf_l, V^{\perp}f_l, VVf_l)$.  Note that due to symmetry after restricting to the diagonal, i.e. setting $\theta_x = \theta_y$ and $\phi_x = \phi_y$, we only need to compute the entries on the diagonal and above.  These calculations are done in full detail in Appendix - Section \ref{app_s:cov}.

For the convenience of the reader, we now list the final form of the entries:

\pagebreak

\begin{enumerate} \label{r: entries_list}
\item $a_{11} = P_l(h(x,y))_{|x=y} = P_l(1) $
\item $a_{12} = V_y P_l(h(x,y))_{|x=y}  = 0 $
\item $a_{13} = V_y^{\perp}P_l(h(x,y))_{|x=y}  = 0 $
\item $a_{14} = V_y V_y P_l(h(x,y))_{|x=y}  = -\| V \|_{g(x)}^2 P'_l(1)$
\item $a_{22} = V_x V_y P_l(h(x,y))_{|x=y}  = \| V \|_{g(x)}^2 P'_l(1)$
\item $a_{23} = V_x V_y^{\perp} P_l(h(x,y))_{|x=y}  = 0$
\item  $a_{24} = V_x V_y V_y P_l(h(x,y)_{|x=y}  = $ \begin{eqnarray*}
 &&  \Big( (v_1^x)^2 \frac{\partial v_1^x}{\partial \theta_x} \sin^2 \phi_x + v_1^x \frac{\partial v_1^x}{\partial \phi_x}v_2^x \sin^2 \phi_x + (v_1^x)^2 v_2^x \sin \phi_x \cos \phi_x \\
 &&  + v_1^x v_2^x \frac{\partial v_2^x}{\partial \theta_x} + (v_2^x)^2 \frac{\partial v_2^x}{\partial \phi_x}\Big)P_l'(1) =: \tilde{a}_{24}(\theta_x, \phi_x) P_l'(1)
\end{eqnarray*}
\item $a_{33} = V_x^{\perp} V_y^{\perp} P_l(h(x,y))_{|x=y}  = \| V^{\perp} \|_{g(x)}^2 P'_l(1)$
\item $a_{34} = V^{\perp}_x V_y V_y P_l(h(x,y))_{|x=y}  =$
\begin{eqnarray*} 
&&  \Big( v_2^x v_1^x \frac{\partial v_1^x}{\partial \theta_x} \sin^2 \phi_x + (v_2^x)^2 v_1^x \sin \phi_x \cos \phi_x + (v_2^x)^2\frac{\partial v_1^x}{\partial \phi_x} \sin^2 \phi_x \\
 && + (v_2^x)^2 v_1^x \sin \phi_x \cos \phi_x + (v_1^x)^3 \cos \phi_x \sin^3 \phi_x - (v_1^x)^2 \frac{\partial v_2^x}{\partial \theta_x} \sin^2 \phi_x - v_1^x v_2^x \frac{\partial v_2^x}{\partial \phi_x} \sin^2 \phi_x   \Big)P_l'(1) \\
 && =: \tilde{a}_{34}(\theta_x, \phi_x) P_l'(1)
\end{eqnarray*}
\item $a_{44} = V_x V_x V_y V_y P_l(h(x,y))_{|x=y}  = \frac{3}{8} \|V\|^4_{g(x)}l^4 + \frac{6}{8} \|V\|^4_{g(x)} l^3 + \frac{1}{2}\tilde{a}_{44}^1(\theta_x, \phi_x) l^2 + (\frac{1}{2}\tilde{a}_{44}^1(\theta_x, \phi_x) - \frac{3}{8} \|V \|^4)l $ where 
\begin{eqnarray*}
&& \tilde{a}_{44}^1(\theta_x, \phi_x) =  v_1^x \frac{\partial v_1^x}{\partial \theta_x } v_1^x \frac{\partial v_1^x}{\partial \theta_x} \sin^2 \phi_x  \\
&& + (v_1^x)^4 \sin^2 \phi_x  - (v_1^x)^3 \frac{\partial v_2^x}{\partial \theta_x} \sin \phi_x \cos \phi_x  + (v_1^x)^2 \frac{\partial v_1^x}{\partial \theta_x} v_2^x \sin \phi_x \cos \phi_x  \\
&& + v_1^x \frac{\partial v_1^x}{\partial \theta_x} \frac{\partial v_1^x}{\partial \phi_x} v_2^x \sin^2 \phi_x  + (v_1^x)^2 \frac{\partial v_1^x}{\partial \theta_x} v_2^x \sin \phi_x \cos \phi_x  - (v_1^x)^2 v_2^x \frac{\partial v_2^x}{\partial \phi_x} \sin \phi_x \cos \phi_x   \\
&&  + (v_1^x)^2 (v_2^x)^2 \sin^2 \phi_x  + (v_1^x)^2 \frac{\partial v_1^x}{\partial \theta_x} v_2^x \cos \phi_x \sin \phi_x  \\
&& - (v_1^x)^3 \frac{\partial v_2^x}{\partial \theta_x} \cos \phi_x \sin \phi_x  + (v_1^x)^2 \left(\frac{\partial v_2^x}{\partial \theta_x}\right)^2   \\
&& + (v_1^x)^2(v_2^x)^2 \cos^2 \phi_x  + v_1^x \frac{\partial v_1^x}{\partial \phi_x}(v_2^x)^2 \cos \phi_x \sin \phi_x    \\
&& + (v_1^x)^2 (v_2^x)^2 \cos^2 \phi_x   + v_1^x \frac{\partial v_2^x}{\partial \theta_x} v_2^x \frac{\partial v_2^x}{\partial \phi_x} \\
&& + v_2^x \frac{\partial v_1^x}{\partial \phi_x} v_1^x \frac{\partial v_1^x}{\partial \theta_x} \sin^2 \phi_x  + v_2^x (v_1^x)^2 \frac{\partial v_1^x}{\partial \theta_x} \cos \phi_x \sin \phi_x  \\
&& + (v_1^x)^2 (v_2^x)^2 \cos^2 \phi_x  + v_1^x \frac{\partial v_1^x}{\partial \phi_x} (v_2^x)^2 \sin \phi_x \cos \phi_x  + v_1^x \frac{\partial v_1^x}{\partial \phi_x} (v_2^x)^2 \cos \phi_x \sin \phi_x  \\
&& + \left( \frac{\partial v_1^x}{\partial \phi_x} \right)^2 (v_2^x)^2 \sin^2 \phi_x  + v_1^x \frac{\partial v_1^x}{\partial \phi_x}(v_2^x)^2 \sin \phi_x \cos \phi_x  \\
&& + (v_1^x)^2 (v_2^x)^2 \cos^2 \phi_x  - (v_1^x)^2 v_2^x \frac{\partial v_2^x}{\partial \phi_x} \cos \phi_x \sin \phi_x  \\
&& + (v_1^x)^2 (v_2^x)^2 \sin^2 \phi_x   + v_1^x v_2^x \frac{\partial v_2^x}{\partial \phi_x} \frac{\partial v_2^x}{\partial \theta_x}  + \left( \frac{\partial v_2^x}{\partial \phi_x} \right)^2(v_2^x)^2   \\
&& + (v_2^x)^4.
\end{eqnarray*}
\end{enumerate}

\subsection{Conditional covariance matrix} \label{subsec:calc_cond_cov}

In this section, we would like to compute the covariance matrix for $(V^{\perp}f_l, VVf_l)$ conditioned on the random vector $(f_l, Vf_l)$, particularly at the value $(0,0)$.  Given that we started with the Gaussian vector field $(f_l, Vf_l, V^{\perp}f_l, VVf_l)$, we can conveniently apply formulas found in \cite{AT00} Section 1.2 for our desired matrix.

The $4 \times 4$-matrix total covariance matrix whose entries we computed in the previous section allows us to explicit calculate
\begin{eqnarray*}
\underbrace{\begin{pmatrix}
a_{33} & a_{34} \\
a_{43} & a_{44} 
\end{pmatrix}}_{=:M_1}
- \underbrace{\begin{pmatrix}
a_{31} & a_{32} \\
a_{41} & a_{42} 
\end{pmatrix}}_{=:M_2} \cdot
\underbrace{\begin{pmatrix}
a_{11} & a_{12} \\
a_{21} & a_{22} 
\end{pmatrix}^{-1}}_{=:M_3} \cdot
\underbrace{\begin{pmatrix}
a_{13} & a_{14} \\
a_{23} & a_{24} 
\end{pmatrix}}_{=:M_4}.
\end{eqnarray*}
	We note that the probability mass function $\Phi_F(0,0)$ of the random field $F=(f_l, Vf_l)$ is $\frac{1}{(2 \pi)\sqrt{\det C_{11}}}$ where $\det C_{11} = \frac{\| V \|^2}{2} (l^2 + l)$.  For the sake of clarity, let us write out these individual matrices: 
\begin{eqnarray*}
&&M_1 =  \begin{pmatrix}
\| V^{\perp} \|^2_{g(x)} \frac{l^2+l}{2} & \tilde{a}_{34}(\theta_x, \phi_x)\frac{l^2+l}{2} \\
\tilde{a}_{34}(\theta_x,\phi_x) \frac{l^2+l}{2} & \left[ \frac{3}{8} \|V\|^4_{g(x)}l^4 + \frac{6}{8} \|V\|^4_{g(x)} l^3 + \frac{1}{2}\tilde{a}_{44}^1(\theta_x, \phi_x) l^2 + (\frac{1}{2}\tilde{a}_{44}^1(\theta_x, \phi_x) - \frac{3}{8} \|V \|^4)l  \right]
\end{pmatrix} \\
&& M_2 =  \begin{pmatrix}
0 & 0 \\
- \| V \|^2_{g(x)}\frac{l^2+l}{2} & \tilde{a}_{24}(\theta_x,\phi_x)\frac{l^2+l}{2}
\end{pmatrix} \\
&& M_3 = \begin{pmatrix}
1 & 0 \\
0 & \frac{2}{\| V \|^2_g (l^2+l)}
\end{pmatrix} \\
&& M_4 = \begin{pmatrix}
0 & - \| V \|^2_{g(x)} \frac{l^2+l}{2} \\
0 & \tilde{a}_{24}(\theta_x,\phi_x)\frac{l^2+l}{2}
\end{pmatrix}
\end{eqnarray*}
where $\lambda_l^2 = l^2+l$ is our Laplace eigenvalue.  This leads to 
\begin{eqnarray*}
&& M_2 \cdot M_3 \cdot M_4 = \\
&& \begin{pmatrix}
0 & 0 \\
0 & \| V \|^4_{g(x)} (\frac{l^2 + l}{2})^2 + \frac{(\tilde{a}_{24}(\theta_x,\phi_x))^2}{ \| V \|_{g(x)}^2} (\frac{l^2 + l}{2})
\end{pmatrix}.
\end{eqnarray*}
Finally, we obtain the symmetric $4 \times 4$ matrix $M_1 - M_2 M_3 M_4$, whose entries $m_{i,j}$ are the following
\begin{eqnarray*}
&& m_{1,1} = \| V^{\perp} \|^2_{g(x)} \frac{l^2+l}{2} =   \frac{\| V^{\perp} \|^2_{g(x)}}{2} \, l^2 + \mathcal{O}(l) \\
&& m_{1,2} = \tilde{a}_{34}(\theta_x,\phi_x)\frac{l^2+l}{2} = \frac{\tilde{a}_{34}(\theta_x,\phi_x)}{2} \, l^2 + \mathcal{O}(l)\\
&& m_{2,2} =  \left( \frac{3}{8} \|V\|^4_{g(x)} -  \frac{1}{4} \| V \|^4_{g(x)} \right)l^4 + \left( \frac{6}{8} \|V\|^4_{g(x)} - \frac{1}{2}\| V \|^4_{g(x)} \right) l^3 \\
&& + \left( \frac{1}{2}\tilde{a}_{44}^1(\theta_x, \phi_x) - \frac{1}{4} \| V \|^4_{g(x)} - \frac{(\tilde{a}_{24}(\theta_x,\phi_x))^2}{ 2\| V \|_{g(x)}^2} \right) l^2 \\
&& + \left(\frac{1}{2}\tilde{a}_{44}^1(\theta_x, \phi_x) - \frac{3}{8} \|V \|^4 - \frac{(\tilde{a}_{24}(\theta_x,\phi_x))^2}{ 2\| V \|_{g(x)}^2} \right)l  \\ 
&& = \frac{\|V \|^4}{8}l^4 +  \frac{\|V \|^4}{4}l^3 + \frac{1}{\|V\|^2_{g(x)}} \mathcal{O}(l^2).
\end{eqnarray*}
Note that the implicit constants appearing in our big-O notation are uniformly bounded in $\phi_x, \theta_x$.  Let us refer to our resulting conditional covariance matrix, with these particular entries, as
\begin{equation*}
\mathbf{\Delta}_l(\theta_x,\phi_x) := \begin{pmatrix}
m_{1,1} & m_{1,2} \\
m_{1,2} & m_{2,2}
\end{pmatrix}.
\end{equation*}

\subsection{Evaluating the first intensity}

Notice that $\det \mathbf{\Delta}_l(\theta_x, \phi_x)$ is possibly singular (i.e. blows up) when $(v_2^x)^2  + \sin^2 \phi_x (v_1^x)^2 = \langle V, V \rangle_g = 0$ which is equivalent to having $V = 0$.  Given that every continuous vector field on $S^2$ must have at least one zero, it is natural that we place some restrictions on how $V$ vanishes.  However, we will show that having a singular determinant in this sense is not the case via an explicit calculation.

As $\mbox{tr}\left( \mathbf{\Delta}_{l}(\theta_x, \phi_x) \right)$ and $\det \left( \mathbf{\Delta}_{l}(\theta_x, \phi_x) \right)$ play an important role in the some upcoming calculations, we write out these quantities explicitly for sake of reference: 
\begin{eqnarray*}
&&\mbox{tr}\left( \mathbf{\Delta}_{l}(\theta_x, \phi_x) \right) = \left( \frac{3}{8} \|V\|^4_{g(x)} -  \frac{1}{4} \| V \|^4_{g(x)} \right)l^4 + \left( \frac{6}{8} \|V\|^4_{g(x)} - \frac{1}{2}\| V \|^4_{g(x)} \right) l^3 \\
&& + \left( \frac{1}{2}\tilde{a}_{44}^1(\theta_x, \phi_x) - \frac{1}{4} \| V \|^4_{g(x)} - \frac{(\tilde{a}_{24}(\theta_x,\phi_x))^2}{ 2\| V \|_{g(x)}^2} + \frac{1}{2} \| V^{\perp} \|^2 \right) l^2 \\
&& + \left(\frac{1}{2}\tilde{a}_{44}^1(\theta_x, \phi_x) - \frac{3}{8} \|V \|^4 - \frac{(\tilde{a}_{24}(\theta_x,\phi_x))^2}{ 2\| V \|_{g(x)}^2} + \frac{1}{2} \| V^{\perp} \|^2 \right)l.  \\ 
\end{eqnarray*}

Using the following formula,
\begin{eqnarray*}
&& \det\left( \mathbf{\Delta}_{l}(\theta_x, \phi_x) \right) = \\
&& \left( \frac{\| V^{\perp} \|^2_g(x) \cdot \| V \|^4_g(x)}{16} \right)l^6  \\
&&  + \frac{\| V^{\perp} \|^2_g(x)}{2} \left( \frac{3 \|V\|_g^4(x)}{16} \right)l^5 \\
&& + \left( \frac{\| V^{\perp} \|^2_g(x)}{2} \left( \frac{\tilde{a}^1_{44}(\theta_x, \phi_x)}{2} - \frac{ \left( \tilde{a}_{24}(\theta_x, \phi_x) \right)^2}{2 \| V \|^2_g} \right) - \frac{\left( \tilde{a}_{34}(\theta_x, \phi_x) \right)^2}{4}\right)l^4 \\
&& + \left( \frac{\| V^{\perp} \|^2_g(x)}{2} \left( \frac{\tilde{a}^1_{44}(\theta_x, \phi_x)}{2} - \frac{3 \|V\|_g^4(x)}{8} - \frac{ \left( \tilde{a}_{24}(\theta_x, \phi_x) \right)^2}{ 2\| V \|^2_g} \right) - \frac{\left( \tilde{a}_{34}(\theta_x, \phi_x) \right)^2}{2}\right)l^3 \\
&& + \left( \frac{\| V^{\perp} \|^2_g(x)}{2} \left( \frac{\tilde{a}^1_{44}(\theta_x, \phi_x)}{2} - \frac{3 \|V\|_g^4(x)}{8} + \frac{ \left( \tilde{a}_{24}(\theta_x, \phi_x) \right)^2}{2 \| V \|^2_g} \right) - \frac{\left( \tilde{a}_{34}(\theta_x, \phi_x) \right)^2}{4}\right)l^2,
\end{eqnarray*}
we get that $\det\left( \mathbf{\Delta}_{l}(\theta_x, \phi_x) \right) = \frac{\|V\|^4 \| V^{\perp} \|^2}{16}l^6 + \mathcal{O}(l^5)$ and the remainder term is uniformly bounded in $\theta_x, \phi_x$ thanks to $\|V^{\perp} \|_{g(x)} = \| V \|_{g(x)} \sqrt{\det g(x)}$ cancelling out the length factor $\| V \|_{g(x)}$ in the denominators.  We record this observation in the following:

\begin{lem} \label{l:det_condcov}
The conditional covariance has the determinant 
\begin{equation*}
\det \left(\mathbf{\Delta}_l(\theta_x, \phi_x) \right) = \frac{\|V\|^6 \det(g(x))}{16}l^6 + \mathcal{O}(l^5)
\end{equation*}
with a uniformly bounded remainder. Moreover, 
$V(\theta_x, \phi_x)= 0$ for some $(\theta_x, \phi_x)$ if and only if $\det \left(\mathbf{\Delta}_l(\theta_x, \phi_x) \right) = 0$ for the same $(\theta_x, \phi_x)$, uniformly for all $l \geq l_0(V)$.
\end{lem}

Recall $\det^{\perp} DF_l = \frac{1}{ \|V \|_g^2}\cdot \left| (V^{\perp}f_l)(VV f_l) \right|$ from equation (\ref{e:orthog_det}).   And since
\begin{eqnarray} \label{e:1st_intensity_expanded}
\nonumber && \Phi_l(0,0) \, \mathbb{E}\left[ | \det^{\perp} DF_l(x) | \big| f_l(x) = 0, V f_l(x) = 0 \right] = \\
&& \frac{1}{2 \pi\sqrt{\det C_{11}}} \, \frac{1}{ \|V \|_g^2} \cdot  \mathbb{E}\left[ \left| (V^{\perp}f_l)(VV f_l) \right|\big| f_l(x) = 0, V f_l(x) = 0 \right] 
\end{eqnarray}
where $\det C_{11} = \frac{\| V \|^2}{2} (l^2 + l)$, we are now lead to our section's main proposition:

\begin{prop} \label{p:1st_intensity_asymp}
Let $\alpha > 0 $, define the set $U :=\{ x \in S^2: \|V\| \geq  l^{-\alpha} \}$.  Then the 1st intensity satisfies the following asymptotic on $U$:
\begin{equation*}
K_V(x) = \Phi_l(0,0) \, \mathbb{E}\left[\mathbf{1}_{\tilde{\Omega}} \, | \det^{\perp} DF_l(x) | \big| f_l(x) = 0, V f_l(x) = 0 \right] = \frac{\sqrt{2}}{4 \pi^2}  \, l^2 +\mathcal{O}_V(l^{1 + 3\alpha})
\end{equation*}
as $l \rightarrow \infty$, where the remainder terms are uniformly bounded in $x \in U$ but have a dependence on $V$ and possibly its derivatives.  Over $U^{\complement}$, we have $K_V(x) = o_V(l^2)$ where once again the subscript notation of $V$ denotes a bounded dependence on the derivatives of $V$.
\end{prop}

\begin{proof} (of Proposition \ref{p:1st_intensity_asymp})

The crux of our proof is in precisely estimating the Gaussian integral
\begin{equation} \label{e:main_gauss_int}
\frac{1}{2 \pi \sqrt{\det \mathbf{\Delta}_l}} \int_{\R^2} |t_1 t_2| \exp \left( -\frac{1}{2} \overrightarrow{t} \det \mathbf{\Delta_l}^{-1} \overrightarrow{t}  \right) \, dt_1 dt_2.
\end{equation}
where 
\begin{equation*}
\mathbf{\Delta_l}^{-1} = \frac{1}{\det \mathbf{\Delta}_l} 
\begin{pmatrix}
\frac{\| V\|^4}{8}l^4 + \frac{\|V\|^4}{4}l^3 + \frac{1}{\|V\|^2} \mathcal{O}(l^2) & -\frac{\tilde{a_{34}}(\theta_x, \phi_x)}{2}l^2 -\frac{\tilde{a_{34}}(\theta_x, \phi_x)}{2}l \\
-\frac{\tilde{a_{34}}(\theta_x, \phi_x)}{2}l^2 -\frac{\tilde{a_{34}}(\theta_x, \phi_x)}{2}l & \frac{\|V^{\perp}\|^2}{2} l^2 + \frac{\|V^{\perp}\|^2}{2} l\\
\end{pmatrix}.
\end{equation*}

We first prove the bound for the first intensity over $U^{\complement}$.  Performing the sequence of transformations:
\begin{equation*}
\overrightarrow{t} = \sqrt{\det \mathbf{\Delta}_l} \, \overrightarrow{s} \Rightarrow \overrightarrow{s} = \frac{l^2}{\|V\|} (r_1, r_2)
\end{equation*}
we are then left with the quadratic form being the 
\begin{equation*}
\begin{pmatrix}
\frac{\|V\|^6}{8} + \mathcal{O}(l^{-1}) & \|V\|^2 \mathcal{O}(l^{-2}) \\
 \|V\|^2 \mathcal{O}(l^{-2}) & \frac{\sin \phi_x^2}{2} \frac{1}{l^2} + \mathcal{O}l^{-3}).
\end{pmatrix}
\end{equation*}

We can make use of our parameter $\alpha$ which dictates our localization around the vanishing set of $V$ to make this quadratic form ``almost" diagonal with positive eigenvalues.  Hence, we set $\alpha < \frac{1}{3}$ in order to make the top-right entry be $> C l^{-2}$.  Let $\epsilon(\alpha) = \frac{1}{3} - \alpha$.  A direct calculation along with Lemma \ref{l:det_condcov} then shows us that equation (\ref{e:main_gauss_int}) is 
\begin{equation*}
\mathcal{O}(\|V\|^{13} l^{5 - 4 \epsilon(\alpha)}) = \mathcal{O}(l^{-13 \alpha} l^{5 - 4 \epsilon(\alpha)})
\end{equation*}
which itself is little-o of $l^2$ if $\alpha > \frac{5}{54} = \frac{1}{11} + \frac{5}{2970}$.

Let us localize away from the zero set $\{V=0\}$ of our vector field $V$, specifically onto the set $U =\{ x \in S^2: \|V\| \geq  l^{-\alpha} \}$ where $\alpha$ is as above.  Here, $l \geq l_0(V)$ as established in Lemma \ref{l:det_condcov}.  Working over this subset of $S^2$ allows us to perform our usual algebraic manipulations in the calculation to follow.

Thanks to the explicit form of the inverse of our conditional covariance matrix, we can perform the following sequence of transformations
\begin{equation*}
\overrightarrow{t} = \sqrt{\det \mathbf{\Delta}_l} \, \overrightarrow{s} \Rightarrow \overrightarrow{s} = \left(\frac{r_1}{l^2}, \frac{r_2}{l} \right) \Rightarrow \overrightarrow{r} = \left( \frac{2 \sqrt{2}}{\|V\|^2} u_1, \frac{\sqrt{2}}{\|V^{\perp}\|} u_2 \right)
\end{equation*}
to reduce the quadratic form appearing in the exponential to 
\begin{eqnarray*} \label{e:covar_decomp}
&& \begin{pmatrix}
1 +  \frac{2}{l} + \frac{1}{\|V\|^6} \mathcal{O}(\frac{1}{l^2}) & -\frac{2 \tilde{a_{34}}(\theta_x, \phi_x)}{\|V\|^2 \|V^{\perp}\|}\frac{1}{l} - \frac{2 \tilde{a_{34}}(\theta_x, \phi_x)}{\|V\|^2 \|V^{\perp}\|}\frac{1}{l^2} \\
-\frac{2 \tilde{a_{34}}(\theta_x, \phi_x)}{\|V\|^2 \|V^{\perp}\|}\frac{1}{l} - \frac{2 \tilde{a_{34}}(\theta_x, \phi_x)}{\|V\|^2 \|V^{\perp}\|}\frac{1}{l^2} & 1 + \frac{1}{l}
\end{pmatrix} \\
\nonumber && = \begin{pmatrix}
1 + \frac{2}{l} & 0 \\
0 & 1 + \frac{1}{l}
\end{pmatrix} + 
\begin{pmatrix}
\frac{1}{\|V\|^6} \mathcal{O}(\frac{1}{l^2}) & -\frac{2 \tilde{a_{34}}(\theta_x, \phi_x)}{\|V\|^2 \|V^{\perp}\|}\frac{1}{l} - \frac{2 \tilde{a_{34}}(\theta_x, \phi_x)}{\|V\|^2 \|V^{\perp}\|}\frac{1}{l^2} \\
-\frac{2 \tilde{a_{34}}(\theta_x, \phi_x)}{\|V\|^2 \|V^{\perp}\|}\frac{1}{l} - \frac{2 \tilde{a_{34}}(\theta_x, \phi_x)}{\|V\|^2 \|V^{\perp}\|}\frac{1}{l^2} & 0
\end{pmatrix}
\end{eqnarray*}
The quadratic polynomial $q(u_1, u_2)$ in $u_1, u_2$ which is generated by this sum of matrices is
\begin{eqnarray*}
&& q(u_1, u_2) = \left( 1 + \frac{2}{l} + \frac{1}{\|V\|^6} \mathcal{O}\left(\frac{1}{l^2} \right) \right)u_1^2 + \left( 1 + \frac{1}{l} \right)u_2^2 \\
&& + \left( -\frac{4 \tilde{a_{34}}(\theta_x, \phi_x)}{\|V\|^2 \|V^{\perp}\|}\frac{1}{l} - \frac{4 \tilde{a_{34}}(\theta_x, \phi_x)}{\|V\|^2 \|V^{\perp}\|}\frac{1}{l^2} \right)u_1 u_2.
\end{eqnarray*}

On the set $U \subset S^2$, in the region $|u_1| < 2|u_2|$, $q(u_1, u_2) = u_1^2 + (1 + \mathcal{O}(l^{-(1 -3\alpha)})) u_2^2$ which follows from the assumption that $V$ must not be allowed to become too small in $l$ which in turn requires that $\frac{\tilde{a_{34}}(\theta_x, \phi_x)}{\|V\|^2 \|V^{\perp}\| l} = \mathcal{O}(\frac{1}{\|V\|^3 l}) =  \mathcal{O} \left( \frac{1}{l^{1-3\alpha}} \right)$; the uniform boundedness in $\theta_x, \phi_x)$ follows as $\frac{\tilde{a_{34}}(\theta_x, \phi_x)}{\sin \phi_x} \in C^{\infty}(S^2)$ after using the explicit form of the entries of the conditional covariance matrix as given in Remark \ref{r: entries_list}.  

In the complementary region $|u_1| \geq 2|u_2|$, we have that $q(u_1, u_2) = \left( 1 + \mathcal{O}(l^{-(1-3\alpha)}) \right)u_1^2 + (1 + l^{-1}) u_2$.  Therefore, on all of $\R^2$, we have that $q(u_1, u_2) = \left( 1 + \mathcal{O}(l^{-(1-3\alpha)}) \right)u_1^2 + (1 + \mathcal{O}(l^{-(1 -3\alpha)})) u_2^2$.  Thus, after a series of reductions, we are left to estimating the quantity
\begin{eqnarray*}
&& \frac{(\det \mathbf{\Delta}_l)^{3/2}}{\pi} \,  \frac{8}{\|V\|^4 \|V^{\perp}\|^2} \, l^{-6} \int_{\R^2} |u_1 u_2| \exp \left(-\frac{1}{2}q(u_1, u_2) \right) \, du_1 \, du_2 \\
&& =  \frac{(\det \mathbf{\Delta}_l)^{3/2}}{\pi} \,  \frac{8}{\|V\|^4 \|V^{\perp}\|^2} \, l^{-6} \, \times \\
&&  \int_{\R^2} |u_1 u_2| \exp \left(-\frac{1}{2}\left( \left( 1 + \mathcal{O}(l^{-(1-3\alpha)}) \right)u_1^2 + (1 + \mathcal{O}(l^{-(1 -3\alpha)})) u_2^2 \right) \right) \, du_1 \, du_2.
\end{eqnarray*}

Hence, our leading term in our asymptotic will come from the expression
\begin{equation*}
\frac{(\det \mathbf{\Delta}_l)^{3/2}}{\pi} \,  \frac{8}{\|V\|^4 \|V^{\perp}\|^2} \, l^{-6} \int_{R^2} |u_1 u_2| \exp \left(-\frac{1}{2} \left( u_1^2 + u_2^2\right) \right) \, du_1 \, du_2.
\end{equation*}
Using that $\det\left( \mathbf{\Delta}_{l}(\theta_x, \phi_x) \right) = \frac{\|V\|^4 \| V^{\perp} \|^2}{16}l^6 + \mathcal{O}(l^5)$, the simple calculation 
\begin{equation*}
\int_{R^2} |u_1 u_2| \exp \left(-\frac{1}{2}(u_1^2 + u_2^2) \right) \, du_1 \, du_2 = \left( 2 \int_0^{\infty} - \frac{d}{du_1} \exp \left(-\frac{1}{2}u_1^2 \right) du_1 \right)^2 = 4,
\end{equation*}
and the remaining factors in equation (\ref{e:1st_intensity_expanded}) for the 1st intensity, we find that the asymptotic of the 1st intensity (on our neighborhood $U$) equals
\begin{eqnarray*}
&& = \Phi_l(0,0) \, \mathbb{E}\left[ | \det^{\perp} DF_l(x) | \big| f_l(x) = 0, V f_l(x) = 0 \right] \\
&& = \frac{\sqrt{2}}{4 \pi^2} \, l^2 + \mathcal{O}_V \left( l^{1+3\alpha} \right)
\end{eqnarray*}
where the big-O terms have implicit constants which are uniformly bounded in $\theta_x, \phi_x$. 

\begin{rem}
We observe throughout our calculation that $\frac{5}{54} < \alpha < 1/3$ is necessary for suitable remainder terms.
\end{rem}

\end{proof}

\subsection{Proof of Theorem \ref{t:main}} \label{subsec:proof_main}

We conclude our note with the proof of the proposed expected value asymptotic:

\begin{proof}
Given Proposition \ref{p:1st_intensity_asymp}, we can now proceed to integrating over $S^2 - \{V=0\}$; recall that $dV_g$ is normalized.  Note also that we now have an approximate expression for the truncated first intensity of the form:
\begin{equation*}
\left(  \frac{\sqrt{2}}{4 \pi^2} \, l^2 + \mathcal{O}_V \left( l^{1+3\alpha} \right) \right) \mathbf{1}_{U}(x) + \left(o_V(l^2) \right)\mathbf{1}_{U^{\complement}}(x)
\end{equation*}

In the case of $\{V = 0 \}$ being finite, we know that the volume of each zero's neighborhood is asymptotic to $l^{-2\frac{\alpha}{m}}$ where $m$ is the smallest order amongst all the zeroes of $V$.  Comparing the integrated bound from near the zeroes of $l^{2 - 2\frac{\alpha}{m}} \log(l)$ (and then summing over the zeroes) and the corresponding bound from away of $\frac{\sqrt{2}}{4 \pi^2} \, l^2 + \mathcal{O} \left( l^{1+3\alpha} \right) + \mathcal{O} \left( l^{2-2\frac{\alpha}{m}} \right) $, we find that setting $\frac{5}{54} < \alpha = \frac{1}{3 + \frac{2}{m}} < 1/3$ gives us our best remainder estimate.

In the second case of $\{ V=0\}$ being a smooth curve of finite length (thanks to 0 being a regular value of $V$ and us being in a compact setting) and $V$ vanishes to order $m$ in the normal directions, we obtain an integrated bound from near the zeroes of $l^{2-\frac{\alpha}{m}}$; proceeding similarly as above we obtain that $\frac{5}{54} < \alpha = \frac{1}{3 + \frac{1}{m}} < 1/3$ gives the best remainder estimate.
\end{proof}

\begin{appendix}

\section{Computation of Covariance Matrix} \label{app_s:cov}
Before we begin, let us remind ourselves that $P_l(t)$ is the standard $l$-th degree Legendre polynomial and $h(x,y) : = \tilde{h}(\theta_x, \theta_y, \phi_x, \phi_y) = \cos \phi_x \cos \phi_y + \sin \phi_x \sin \phi_y \cos( \theta_x - \theta_y)$.  We remind ourselves of the coordinate representations $V_y = v_1(\theta_y, \phi_y) \frac{\partial}{\partial \theta_y} + v_2(\theta_y, \phi_y) \frac{\partial}{\partial \phi_y}$ and $V^{\perp}_y = v_2(\theta_y, \phi_y) \frac{\partial}{\partial \theta_y} - \sin^2 \phi_y v_1(\theta_y, \phi_y) \frac{\partial}{\partial \phi_y}$.  Similar definitions hold for $V_x$ and $V_x^{\perp}$.

For the ease of exposition, let us list the specific formulas for the entries of the full covariance matrix for the random field $(f_l, Vf_l, V^{\perp}f_l, VVf_l) \in \R^4$:

\begin{itemize}
\item $a_{11} = P_l(h(x,y))_{|x=y}$
\item $a_{12} = V_y P_l(h(x,y))_{|x=y}$
\item $a_{13} = V_y^{\perp}P_l(h(x,y))_{|x=y}$
\item $a_{14} = V_y V_y P_l(h(x,y))_{|x=y}$
\item $a_{22} = V_x V_y P_l(h(x,y))_{|x=y}$
\item $a_{23} = V_x V_y^{\perp} P_l(h(x,y))_{|x=y}$
\item $a_{24} = V_x V_y V_y P_l(h(x,y))_{|x=y}$
\item $a_{33} = V_x^{\perp} V_y^{\perp} P_l(h(x,y))_{|x=y}$
\item $a_{34} = V^{\perp}_x V_y V_y P_l(h(x,y))_{|x=y}$
\item $a_{44} = V_x V_x V_y V_y P_l(h(x,y))_{|x=y}$
\end{itemize}
We now proceed to calculating these entries along the diagonal:

\noindent \underline{$a_{11}$}: $$P_l(h)_{\mid\theta_x = \theta_y,  \phi_x = \phi_y} = P_l(1)$$

\noindent  \underline{$a_{12}$}: \begin{eqnarray*}
&& v_1^y P_l'(h) (\sin \phi_x \sin \phi_y \sin(\theta_x - \theta_y)) \\
&& + v_2^y P_l'(h) (- \cos \phi_x \sin \phi_y + \sin \phi_x \cos \phi_y cos(\theta_x - \theta_y)_{\mid\theta_x = \theta_y,  \phi_x = \phi_y} = 0
\end{eqnarray*} 

\noindent \underline{$a_{13}$}: 
\begin{eqnarray*} 
&& v_2^y P_l'(h) (\sin \phi_x \sin \phi_y \sin(\theta_x - \theta_y)) + (-\sin^2 \phi_y)v_1^y P_l'(h) (- \cos \phi_x \sin \phi_y \\
&&  + \sin \phi_x \cos \phi_y cos(\theta_x - \theta_y))_{\mid\theta_x = \theta_y,  \phi_x = \phi_y} = 0
\end{eqnarray*}

\noindent \underline{$a_{14}$}: We set $T_1 : = v_1^y(\theta_y,\phi_y) \frac{\partial}{\partial \theta_y}$ and $T_2 :=  v_2^y(\theta_y,\phi_y) \frac{\partial}{\partial \phi_y}$.  This will allow us to organize our derivative calculations more easily.  We continue working with the general formula, pre-evaluation at $x=y$, appearing for $a_{12}$.  We break this calculation into blocks arising from different applications of the vector fields $T_1$ and $T_2$.

\begin{eqnarray*}
\mbox{block 1} = T_1(S_1): && v_1^y \frac{\partial v_1^y}{\partial \theta_y} P_l'(h) (\sin \phi_x \sin \phi_y \sin (\theta_x - \theta_y)) + (v_1^y)^2 P_l''(h) (\sin \phi_x \sin \phi_y \sin (\theta_x - \theta_y)) ^2 \\
&& + (v_1^y)^2 P_l'(h) (-\sin \phi_x \sin \phi_y \cos (\theta_x - \theta_y)) 
\end{eqnarray*}
\begin{eqnarray*}
+ \mbox{block 2} = T_1(S_2): && v_1^y \frac{\partial v_2^y}{\partial \theta_y}  P_l ' (h) (- \cos \phi_x \sin \phi_y  + \sin \phi_x \cos \phi_y \cos(\theta_x - \theta_y)) + \\
&& v_1^y v_2^y P_l''(h) (\sin \phi_x \sin \phi_y \sin (\theta_x - \theta_y))(- \cos \phi_x \sin \phi_y  + \sin \phi_x \cos \phi_y \cos(\theta_x - \theta_y)) + \\
&& v_1^y v_2^y P_l'(h) (\sin \phi_x \cos \phi_y \sin(\theta_x - \theta_y)) \end{eqnarray*}
\begin{eqnarray*}
+ \mbox{block 3 } = T_2(S_1): && \frac{\partial v_1^y}{\partial \phi_y} v_2^y P_l'(h) (\sin \phi_x \sin \phi_y \sin(\theta_x - \theta_y) ) + \\
&& v_1^y v_2^y P_l''(h) (- \cos \phi_x \sin \phi_y  + \sin \phi_x \cos \phi_y \cos(\theta_x - \theta_y)) (\sin \phi_x \sin \phi_y \sin (\theta_x - \theta_y)) + \\
&&  v_1^y v_2^y P_l'(h) (\sin \phi_x \cos \phi_y \sin (\theta_x - \theta_y)) \end{eqnarray*}
\begin{eqnarray*}
+ \mbox{block 4} = T_2(S_2): &&  v_2^y \frac{\partial v_2^y}{\partial \phi_y} P_l'(h) (- \cos \phi_x \sin \phi_y  + \sin \phi_x \cos \phi_y \cos(\theta_x - \theta_y)) + \\
&& (v_2^y)^2 P_l''(h) (- \cos \phi_x \sin \phi_y  + \sin \phi_x \cos \phi_y \cos(\theta_x - \theta_y))^2 + \\
&& (v_2^y)^2 P_l'(h) (- \cos \phi_x \cos \phi_y  - \sin \phi_x \sin \phi_y \cos(\theta_x - \theta_y))_{\mid \theta_x = \theta_y, \phi_x = \phi_y} \end{eqnarray*}
\begin{eqnarray*}
= \mbox{final:} && - (v_1^x)^2 P_l'(1) (\sin^2 \phi_x) - (v_2^x)^2 P_l'(1) = -\| V \|_{g(x)}^2 P'_l(1)
\end{eqnarray*}

\noindent \underline{$a_{22}$}:
\begin{eqnarray*}
\mbox{block 1:} && -v_1^x v_1^y  P_l''(h)(\sin \phi_x \sin \phi_y \sin(\theta_x - \theta_y))^2 + v_1^x v_1^y P_l'(h) \left(\sin\phi_x \sin \phi_y \cos(\theta_x - \theta_y) \right) 
\end{eqnarray*}
\begin{eqnarray*}
+ \mbox{block 2:} && v_1^x v_2^y P_l''(h) (-\sin \phi_x \sin \phi_y \sin(\theta_x - \theta_y)) (- \cos \phi_x \sin \phi_y + \sin \phi_x \cos \phi_y \cos( \theta_x - \theta_y)) + \\
&& v_1^x v_2^y P_l'(h)(- \sin \phi_x \cos \phi_y \sin( \theta_x - \theta_y))  
\end{eqnarray*}
\begin{eqnarray*}
+ \mbox{block 3:} && v_1^y v_2^x P_l''(h) (-\sin \phi_x \cos \phi_y + \cos \phi_x \sin \phi_y \cos(\theta_x - \theta_y))(\sin \phi_x \sin \phi_y \sin (\theta_x - \theta_y)) + \\
&& v_1^y v_2^x P_l'(h) (\cos \phi_x \sin \phi_y \sin(\theta_x - \theta_y)) \end{eqnarray*}
\begin{eqnarray*}
+ \mbox{block 4:} && v_2^x v_2^y P_l''(-\sin\phi_x \cos \phi_y + \cos \phi_y \sin \phi_y \cos (\theta_x - \theta_y)) \\
&& \cdot ( -\cos \phi_x \sin \phi_y + \sin \phi_x \cos \phi_y \cos (\theta_x - \theta_y)) + \\
&& v_2^x v_2^y P_l' (\sin \phi_x \sin \phi_y + \cos \phi_x \cos \phi_y \cos (\theta_x - \theta_y))_{\mid \theta_x = \theta_y, \phi_x = \phi_y} \\\end{eqnarray*}
\begin{eqnarray*}
= \mbox{final:} && (v_1^x)^2 P_l'(1)(\sin^2 \phi_x) + (v_2^x)^2 P_l'(1) = \| V \|^2_{g(x)} P'_l(1)
\end{eqnarray*}

\noindent \underline{$a_{23}$}:
\begin{eqnarray*}
\mbox{block 1:} && - v_1^x v_2^y P_l''(h) (\sin \phi_x \sin \phi_y \sin (\theta_x - \theta_y))^2 + \\
&&  v_1^x v_2^y P_l'(h) (\sin \phi_x \sin \phi_y \cos(\theta_x - \theta_y)) 
\end{eqnarray*}
\begin{eqnarray*}
+ \mbox{block 2:} && - v_1^x (- \sin^2 \phi_y) v_1^y P_l''(h) (\sin \phi_x \sin \phi_y \sin (\theta_x - \theta_y)) \\
&& \cdot (- \cos \phi_x \sin \phi_y + \sin \phi_x \cos \phi_y \cos (\theta_x - \theta_y))  \\
&& - v_1^x (- \sin^2 \phi_y) v_1^y P_l'(h) (\sin \phi_x \cos \phi_y \sin(\theta_x - \theta_y)) 
\end{eqnarray*}
\begin{eqnarray*}
+ \mbox{block 3:} && v_2^x v_2^y P_l''(h) (-\sin \phi_x \cos \phi_y + \cos \phi_x \sin \phi_y \cos(\theta_x - \theta_y))\\
&& \cdot (\sin \phi_x \sin \phi_y \sin (\theta_x - \theta_y)) + v_2^x v_2^y P_l'(h) (\cos \phi_x \sin \phi_y \sin (\theta_x - \theta_y)) 
\end{eqnarray*}
\begin{eqnarray*}
+ \mbox{block 4:} && v_2^x (- \sin^2 \phi_y) v_1^y P_l''(h) (-\sin \phi_x \cos \phi_y + \cos \phi_x \sin \phi_y \cos (\theta_x - \theta_y)) \\ 
&& \cdot (- \cos \phi_x \sin \phi_y + \sin \phi_x \cos \phi_y \cos (\theta_x - \theta_y)) + \\
&& v_2^x (- \sin^2 \phi_y) v_1^y P_l'(h) (\sin \phi_x \sin \phi_y + \cos \phi_x \cos \phi_y \cos (\theta_x - \theta_y))_{\mid \theta_x = \theta_y, \phi_x = \phi_y} \end{eqnarray*}
\begin{eqnarray*}
= \mbox{final:} && v_1^x v_2^x (\sin^2 \phi_x) P_l'(1) + (- \sin^2 \phi_x) v_1^x v_2^x P_l'(1) = 0
\end{eqnarray*}

\noindent \underline{$a_{33}$}:  We continue working with the general formula, pre-evaluation at $x=y$, appearing for $a_{13}$.  We break this calculation into blocks arising from different applications of the components of the vector fields $V^{\perp}$.
\begin{eqnarray*}
\mbox{block 1:} && v_2^x v_2^y P_l''(h) (-\sin \phi_x \sin \phi_y \sin( \theta_x - \theta_y)) (\sin \phi_x \sin \phi_y \sin (\theta_x - \theta_y))  \\
&& + v_2^x v_2^y P_l'(h) (\sin \phi_x \sin \phi_y \cos (\theta_x - \theta_y)) \end{eqnarray*}
\begin{eqnarray*}
+ \mbox{block 2:} && -   \sin^2 \phi_y v_2^x v_1^y P_l''(h) (-\sin\phi_x \sin \phi_y \sin(\theta_x - \theta_y)) (-\cos \phi_x \sin \phi_y + \sin \phi_x \cos \phi_y \cos(\theta_x - \theta_y)) \\
&& -  \sin^2 \phi_y v_2^x v_1^y P_l'(h) (-\sin \phi_x \cos \phi_y \sin (\theta_x - \theta_y)) \end{eqnarray*}
\begin{eqnarray*}
+\mbox{block 3:} && - \sin^2 \phi_x v_1^x v_2^y P_l''(h) (- \sin \phi_x \cos \phi_y + \cos \phi_x \sin \phi_y \cos(\theta_x - \theta_y))(\sin \phi_x \sin \phi_y \sin (\theta_x - \theta_y))  \\
&& - \sin^2 \phi_x v_1^x v_2^y P_l'(h) (\cos \phi_x \sin \phi_y \sin (\theta_x - \theta_y)) \end{eqnarray*}
\begin{eqnarray*}
+ \mbox{block 4:} &&  (-\sin^2 \phi_x) (-\sin^2 \phi_y)v_1^x v_1^y P_l''(h) (-\sin \phi_x \cos \phi_y + \cos \phi_x \sin \phi_y \cos (\theta_x - \theta_y)) \\
&& \cdot (-\cos \phi_x \sin \phi_y + \sin \phi_x \cos \phi_y \cos(\theta_x - \theta_y)) \\
&& + (-\sin^2 \phi_x) (-\sin^2 \phi_y) v_1^x v_1^y P_l'(h) (\sin \phi_x \sin \phi_y + \cos \phi_x \cos \phi_y \cos(\theta_x - \theta_y))_{\mid \theta_x = \theta_y, \phi_x = \phi_y} \end{eqnarray*}
\begin{eqnarray*}
= \mbox{final:} && \sin^2 \phi_x (v_2^x)^2 P_l'(1)+ \sin^4 \phi_x (v_1^x)^2 P_l'(1) = \| V^{\perp} \|^2_{g(x)} P_l'(1)
\end{eqnarray*}

\noindent \underline{$a_{24}$}:  Let $T_1 = v_1^x \frac{\partial}{\partial \theta_x}$ and $T_2 = v_2^x \frac{\partial}{\partial \phi_x}$.  We continue working with the general formula, pre-evaluation at $x=y$, appearing for $a_{14}$.  We break this calculation into blocks arising from different applications of the vector fields $T_1$ and $T_2$.
\begin{eqnarray*}
\mbox{$T_1$ (block 1), S1:} && v_1^x v_1^y \frac{\partial v_1^y}{\partial \theta_y} P_l''(h) (-\sin \phi_x \sin \phi_y \sin (\theta_x - \theta_y)) (\sin \phi_x \sin \phi_y \sin (\theta_x - \theta_y)) \\
&& + v_1^x v_1^y \frac{\partial v_1^y}{\partial \theta_y} P_l'(h) (\sin \phi_x \sin \phi_y \cos (\theta_x - \theta_y)) \end{eqnarray*}
\begin{eqnarray*}
+ \mbox{$T_1$ (block 1), S2:} && v_1^x (v_1^y)^2 P_l '''(h) (-\sin \phi_x \sin \phi_y \sin(\theta_x - \theta_y))(\sin \phi_x \sin \phi_y \sin(\theta_x - \theta_y))^2 \\
&& + 2v_1^x (v_1^y)^2 P_l ''(h)(\sin \phi_x \sin \phi_y \sin(\theta_x - \theta_y))(\sin \phi_x \sin \phi_y \cos(\theta_x - \theta_y)) \end{eqnarray*}
\begin{eqnarray*}
+ \mbox{$T_1$ (block 1), S3:} && v_1^x (v_1^y)^2 P_l ''(h)(- \sin \phi_x \sin \phi_y \sin(\theta_x - \theta_y))(- \sin \phi_x \sin \phi_y \cos(\theta_x - \theta_y)) \\
&& +  v_1^x (v_1^y)^2 P_l '(h)(\sin \phi_x \sin \phi_y \sin(\theta_x - \theta_y)) \end{eqnarray*}
\begin{eqnarray*}
+ \mbox{$T_1$ (block 2), S1:} && v_1^x v_1^y \frac{\partial v_2^y}{\partial \theta_y}P_l''(h) (-\sin \phi_x \sin \phi_y \sin(\theta_x - \theta_y)) \\
&& \cdot  (- \cos \phi_x \sin \phi_y + \sin \phi_x \cos \phi_y \cos (\theta_x - \theta_y)) \\
&& + v_1^x v_1^y \frac{\partial v_2^y}{\partial \theta_y}P_l'(h) (- \sin \phi_x \cos \phi_y \sin (\theta_x - \theta_y)) \end{eqnarray*}
\begin{eqnarray*}
+ \mbox{$T_1$ (block 2), S2:} && v_1^x v_1^y v_2^y P_l'''(h) (- \sin \phi_x \sin \phi_y \sin (\theta_x - \theta_y)) ( \sin \phi_x \sin \phi_y \sin (\theta_x - \theta_y)) \\
&& \cdot  (- \cos \phi_x \sin \phi_y + \sin \phi_x \cos \phi_y \cos (\theta_x - \theta_y)) \\
&& + v_1^x v_1^y v_2^y P_l''(h)(\sin \phi_x \sin \phi_y \cos (\theta_x - \theta_y)) (- \cos \phi_x \sin \phi_y + \sin \phi_x \cos \phi_y \cos (\theta_x - \theta_y))  \\
&& +  v_1^x v_1^y v_2^y P_l''(h)(\sin \phi_x \sin \phi_y \sin (\theta_x - \theta_y)) (- \sin \phi_x \cos \phi_y \sin (\theta_x - \theta_y)) \end{eqnarray*}
\begin{eqnarray*}
+ \mbox{$T_1$ (block 2), S3:} && v_1^x v_1^y v_2^y P_l''(h) (-\sin \phi_x \sin \phi_y \sin(\theta_x - \theta_y)) (\sin \phi_x \cos \phi_y \sin (\theta_x - \theta_y)) \\
&& + v_1^x v_1^y v_2^y P_l'(h) (\sin \phi_x \cos \phi_y \cos (\theta_x - \theta_y)) \end{eqnarray*}
\begin{eqnarray*}
+ \mbox{$T_1$ (block 3), S1:} && v_1^x \frac{\partial v_1^y}{\partial \phi_y} v_2^yP_l''(h)(-\sin \phi_x \sin \phi_y \sin(\theta_x - \theta_y))(\sin \phi_x \sin \phi_y \sin (\theta_x - \theta_y)) + \\
&& v_1^x \frac{\partial v_1^y}{\partial \phi_y} v_2^yP_l'(h)(\sin \phi_x \sin \phi_y \cos(\theta_x - \theta_y)) \end{eqnarray*}
\begin{eqnarray*}
+ \mbox{$T_1$ (block 3), S2:} && v_1^x v_1^y v_2^y P_l'''(h)(-\sin \phi_x \sin \phi_y \sin(\theta_x - \theta_y)) \\
&& \cdot  (- \cos \phi_x \sin \phi_y + \sin \phi_x \cos \phi_y \cos (\theta_x - \theta_y))(\sin \phi_x \sin \phi_y \sin (\theta_x - \theta_y))   \\
&&+  v_1^x v_1^y v_2^y P_l''(h)(-\sin \phi_x \cos \phi_y \sin(\theta_x - \theta_y))(\sin \phi_x \sin \phi_y \sin(\theta_x - \theta_y)) \\
&& + v_1^x v_1^y v_2^y P_l''(h) (- \cos \phi_x \sin \phi_y + \sin \phi_x \cos \phi_y \cos (\theta_x - \theta_y)) \\
&& \cdot (\sin \phi_x \sin \phi_y \cos (\theta_x - \theta_y)) \end{eqnarray*}
\begin{eqnarray*}
+ \mbox{$T_1$ (block 3), S3:} && + v_1^x v_1^y v_2^y P_l''(h) (-\sin \phi_x \sin \phi_y \sin(\theta_x - \theta_y))(\sin \phi_x \cos \phi_y \sin (\theta_x - \theta_y))  \\
&& + v_1^x v_1^y v_2^y P_l'(h)(\sin \phi_x \cos \phi_y \cos (\theta_x - \theta_y)) \end{eqnarray*}
\begin{eqnarray*}
+ \mbox{$T_1$ (block 4), S1:} && v_1^x v_2^y \frac{\partial v_2^y}{\partial \phi_y}P_l''(h) (-\sin \phi_x \sin \phi_y \sin(\theta_x - \theta_y)) (-\cos \phi_x \sin \phi_y \\
&& + \sin \phi_x \cos \phi_y \cos(\theta_x - \theta_y)) \\
&&+  v_1^x v_2^y \frac{\partial v_2^y}{\partial \phi_y}P_l'(h)(-\sin \phi_x \cos \phi_y \sin(\theta_x - \theta_y)) \end{eqnarray*}
\begin{eqnarray*}
+ \mbox{$T_1$ (block 4), S2:} &&  v_1^x(v_2^y)^2 P_l'''(h)(-\sin \phi_x \sin \phi_y \sin(\theta_x - \theta_y) ) \\
&& \cdot (- \cos \phi_x \sin \phi_y  + \sin \phi_x \cos \phi_y \cos(\theta_x - \theta_y))^2 \\
&& + 2v_1^x(v_2^y)^2 P_l''(h)(- \cos \phi_x \sin \phi_y  + \sin \phi_x \cos \phi_y \cos(\theta_x - \theta_y)) \\
&& \cdot (-\sin \phi_x \cos \phi_y \sin(\theta_x - \theta_y)) 
\end{eqnarray*}

\begin{eqnarray*}
\mbox{$T_1$ (block 4), S3:} &&  v_1^x(v_2^y)^2 P_l''(h) (-\sin \phi_x \sin \phi_y \sin(\theta_x - \theta_y)) (- \cos \phi_x \cos \phi_y  - \sin \phi_x \sin \phi_y \\
&& \cdot \cos(\theta_x - \theta_y)) \\
&& + v_1^x(v_2^y)^2 P_l'(h) (\sin \phi_x \sin \phi_y \sin(\theta_x - \theta_y)) \end{eqnarray*}
\begin{eqnarray*}
+ \mbox{$T_2$ (block 1), S1:} && v_2^x v_1^y \frac{\partial v_1^y}{\partial \theta_y} P_l''(h)(-\sin \phi_x \cos \phi_y + \cos \phi_x \sin \phi_y \cos (\theta_x - \theta_y)) \\
&& \cdot (\sin \phi_x \sin \phi_y \sin (\theta_x - \theta_y)) + v_2^x v_1^y \frac{\partial v_1^y}{\partial \theta_y} P_l'(h) (\cos \phi_x \sin \phi_y \sin (\theta_x - \theta_y)) \end{eqnarray*}
\begin{eqnarray*}
+ \mbox{$T_2$ (block 1), S2:} && v_2^x (v_1^y)^2 P_l'''(h) (-\sin \phi_x \cos \phi_y + \cos \phi_x \sin \phi_y \cos (\theta_x - \theta_y)) \\
&& \cdot (\sin \phi_x \sin \phi_y \sin (\theta_x - \theta_y))^2 \\
&& + 2v_2^x (v_1^y)^2 P_l''(h)(\sin \phi_x \sin \phi_y \sin (\theta_x - \theta_y))(\cos \phi_x \sin \phi_y \sin (\theta_x - \theta_y)) \end{eqnarray*}
\begin{eqnarray*}
+ \mbox{$T_2$ (block 1), S3:} && v_2^x (v_1^y)^2 P_l''(h) (-\sin \phi_x \cos \phi_y + \cos \phi_x \sin \phi_y \cos (\theta_x - \theta_y)) \\
&& \cdot (-\sin \phi_x \sin \phi_y \cos (\theta_x - \theta_y)) \\
&& +  v_2^x (v_1^y)^2 P_l'(h)(- \cos \phi_x \sin \phi_y \cos (\theta_x - \theta_y)) \end{eqnarray*}
\begin{eqnarray*}
+ \mbox{$T_2$ (block 2), S1:} && v_2^x v_1^y \frac{\partial v_2^y}{\partial \theta_y}  P_l '' (h)(-\sin \phi_x \cos \phi_y + \cos \phi_x \sin \phi_y \cos (\theta_x - \theta_y)) \\
&& \cdot (- \cos \phi_x \sin \phi_y  + \sin \phi_x \cos \phi_y \cos(\theta_x - \theta_y)) \\
&& + v_2^x v_1^y \frac{\partial v_2^y}{\partial \theta_y}  P_l ' (h) \cdot (\sin \phi_x \sin \phi_y  + \cos \phi_x \cos \phi_y \cos(\theta_x - \theta_y)) \end{eqnarray*}
\begin{eqnarray*}
+ \mbox{$T_2$ (block 2), S2:} && v_2^x v_1^y v_2^y P_l'''(h)  (-\sin \phi_x \cos \phi_y + \cos \phi_x \sin \phi_y \cos (\theta_x - \theta_y))   \\
&& \cdot (\sin \phi_x \sin \phi_y \sin (\theta_x - \theta_y))(- \cos \phi_x \sin \phi_y  + \sin \phi_x \cos \phi_y \cos(\theta_x - \theta_y)) \\
&& + v_2^x v_1^y v_2^y P_l''(h)(\cos \phi_x \sin \phi_y \sin (\theta_x - \theta_y))(- \cos \phi_x \sin \phi_y  + \sin \phi_x \cos \phi_y \cos(\theta_x - \theta_y)) \\
&& + v_2^x v_1^y v_2^y P_l''(h) (\sin \phi_x \sin \phi_y \sin (\theta_x - \theta_y))(\sin \phi_x \sin \phi_y  + \cos \phi_x \cos \phi_y \cos(\theta_x - \theta_y)) \end{eqnarray*}
\begin{eqnarray*}
+ \mbox{$T_2$ (block 2), S3:} && v_2^x v_1^y v_2^y P_l''(h) (- \sin \phi_x \cos \phi_y  + \cos \phi_x \sin \phi_y \cos(\theta_x - \theta_y)) (\sin \phi_x \cos \phi_y \sin(\theta_x - \theta_y)) \\
&& + v_2^x v_1^y v_2^y P_l'(h) (\cos \phi_x \cos \phi_y \sin (\theta_x - \theta_y)) \end{eqnarray*}
\begin{eqnarray*}
+ \mbox{$T_2$ (block 3), S1:} && v_2^x \frac{\partial v_1^y}{\partial \phi_y} v_2^y P_l''(h)(- \sin \phi_x \cos \phi_y  + \cos \phi_x \sin \phi_y \cos(\theta_x - \theta_y)) \\
&& \cdot (\sin \phi_x \sin \phi_y \sin(\theta_x - \theta_y) ) \\
&& + v_2^x \frac{\partial v_1^y}{\partial \phi_y} v_2^y P_l'(h) (\cos \phi_x \sin \phi_y \sin(\theta_x - \theta_y)) \end{eqnarray*}
\begin{eqnarray*}
+ \mbox{$T_2$ (block 3), S2:} && v_1^y v_2^x v_2^y P_l'''(h)(- \sin \phi_x \cos \phi_y  + \cos \phi_x \sin \phi_y \cos(\theta_x - \theta_y)) \\
&& \cdot (- \cos \phi_x \sin \phi_y  + \sin \phi_x \cos \phi_y \cos(\theta_x - \theta_y)) (\sin \phi_x \sin \phi_y \sin (\theta_x - \theta_y))  \\
&& + v_1^y v_2^x v_2^y P_l''(h) (\sin \phi_x \sin \phi_y  + \cos \phi_x \cos \phi_y \cos(\theta_x - \theta_y)) \\ 
&& \cdot (\sin \phi_x \sin \phi_y \sin (\theta_x - \theta_y)) \\
&& + v_1^y v_2^x v_2^y P_l''(h) (-\cos \phi_x \sin \phi_y  + \sin \phi_x \cos \phi_y \cos(\theta_x - \theta_y)) \\ 
&& \cdot (\cos \phi_x \sin \phi_y \sin (\theta_x - \theta_y))
\end{eqnarray*}

\begin{eqnarray*}
\mbox{$T_2$ (block 3), S3:} &&  v_2^x v_1^y v_2^y P_l''(h) (- \sin \phi_x \cos \phi_y  + \cos \phi_x \sin \phi_y \cos(\theta_x - \theta_y)) (\sin \phi_x \cos \phi_y \\
&& \cdot \sin (\theta_x - \theta_y))  \\
&& +  v_2^x v_1^y v_2^y P_l'(h)(\cos \phi_x \cos \phi_y \sin (\theta_x - \theta_y)) \end{eqnarray*}
\begin{eqnarray*}
+ \mbox{$T_2$ (block 4), S1:} && v_2^x v_2^y \frac{\partial v_2^y}{\partial \phi_y} P_l''(h)  (- \sin \phi_x \cos \phi_y  + \cos \phi_x \sin \phi_y \cos(\theta_x - \theta_y)) (- \cos \phi_x \sin \phi_y  + \\
&& \sin \phi_x \cos \phi_y \cos(\theta_x - \theta_y)) \\
&& + v_2^x v_2^y \frac{\partial v_2^y}{\partial \phi_y} P_l'(h)(\sin \phi_x \sin \phi_y  + \cos \phi_x \cos \phi_y \cos(\theta_x - \theta_y)) \end{eqnarray*}
\begin{eqnarray*}
+ \mbox{$T_2$ (block 4), S2:} && v_2^x (v_2^y)^2 P_l'''(h)  (- \sin \phi_x \cos \phi_y  + \cos \phi_x \sin \phi_y \cos(\theta_x - \theta_y)) (- \cos \phi_x \sin \phi_y  \\
&& + \sin \phi_x \cos \phi_y \cos(\theta_x - \theta_y))^2 \\
&& + 2v_2^x (v_2^y)^2 P_l''(h) (- \cos \phi_x \sin \phi_y  + \sin \phi_x \cos \phi_y \cos(\theta_x - \theta_y))(\sin \phi_x \sin \phi_y  \\
&& + \cos \phi_x \cos \phi_y \cos(\theta_x - \theta_y)) \end{eqnarray*}
\begin{eqnarray*}
+ \mbox{$T_2$ (block 4), S3:} && v_2^x (v_2^y)^2 P_l''(h) (- \sin \phi_x \cos \phi_y  + \cos \phi_x \sin \phi_y \cos(\theta_x - \theta_y)) (- \cos \phi_x \cos \phi_y \\
&&   - \sin \phi_x \sin \phi_y \cos(\theta_x - \theta_y)) \\
&& + v_2^x (v_2^y)^2 P_l'(h) (\sin \phi_x \cos \phi_y   - \cos \phi_x \sin \phi_y \cos(\theta_x - \theta_y)) _{| \theta_x = \theta_y, \phi_x = \phi_y}
\end{eqnarray*}
\begin{eqnarray*}
 && = \mbox{final} : \Big( (v_1^x)^2 \frac{\partial v_1^x}{\partial \theta_x} \sin^2 \phi_x + (v_1^x)^2 v_2^x \sin \phi_x \cos \phi_x + v_1^x \frac{\partial v_1^x}{\partial \phi_x}v_2^x \sin^2 \phi_x + (v_1^x)^2 v_2^x \sin \phi_x \cos \phi_x \\
 && - v_2^x(v_1^x)^2 \sin \phi_x \cos \phi_x  + v_1^x v_2^x \frac{\partial v_2^x}{\partial \theta_x} + (v_2^x)^2 \frac{\partial v_2^x}{\partial \phi_x}\Big)P_l'(1)
\end{eqnarray*}
\begin{eqnarray*}
 && = \Big( (v_1^x)^2 \frac{\partial v_1^x}{\partial \theta_x} \sin^2 \phi_x + v_1^x \frac{\partial v_1^x}{\partial \phi_x}v_2^x \sin^2 \phi_x + (v_1^x)^2 v_2^x \sin \phi_x \cos \phi_x \\
 &&  + v_1^x v_2^x \frac{\partial v_2^x}{\partial \theta_x} + (v_2^x)^2 \frac{\partial v_2^x}{\partial \phi_x}\Big)P_l'(1)
\end{eqnarray*}

\noindent \underline{$a_{34}$}: We notice this entry is exactly the same as $a_{24}$ except that for every occurence of $v_1^x$ and $v_2^x$, we substitute with $- \sin^2 \phi_x v_2^x$ and $v_1^x$, respectively.  Let $T_1 =  v_2^x \frac{\partial}{\partial \theta_x}$ and $T_2 = - \sin^2 \phi_x v_1^x \frac{\partial}{\partial \phi_x}$

\begin{eqnarray*}
\mbox{$T_1$ (block 1), S1:} && v_2^x v_1^y \frac{\partial v_1^y}{\partial \theta_y} P_l''(h) (-\sin \phi_x \sin \phi_y \sin (\theta_x - \theta_y)) (\sin \phi_x \sin \phi_y \sin (\theta_x - \theta_y)) \\
&& + v_2^x v_1^y \frac{\partial v_1^y}{\partial \theta_y} P_l'(h) (\sin \phi_x \sin \phi_y \cos (\theta_x - \theta_y)) \end{eqnarray*}
\begin{eqnarray*}
+ \mbox{$T_1$ (block 1), S2:} && v_2^x (v_1^y)^2 P_l '''(h) (-\sin \phi_x \sin \phi_y \sin(\theta_x - \theta_y))(\sin \phi_x \sin \phi_y \sin(\theta_x - \theta_y))^2 \\
&& + 2v_2^x (v_1^y)^2 P_l ''(h)(\sin \phi_x \sin \phi_y \sin(\theta_x - \theta_y))(\sin \phi_x \sin \phi_y \cos(\theta_x - \theta_y)) \end{eqnarray*}
\begin{eqnarray*}
+ \mbox{$T_1$ (block 1), S3:} && v_2^x (v_1^y)^2 P_l ''(h)(- \sin \phi_x \sin \phi_y \sin(\theta_x - \theta_y))\\ 
&& \cdot (- \sin \phi_x \sin \phi_y \cos(\theta_x - \theta_y)) \\
&& +  v_2^x (v_1^y)^2 P_l '(h)(\sin \phi_x \sin \phi_y \sin(\theta_x - \theta_y)) \end{eqnarray*}
\begin{eqnarray*}
+ \mbox{$T_1$ (block 2), S1:} && v_2^x v_1^y \frac{\partial v_2^y}{\partial \theta_y}P_l''(h) (-\sin \phi_x \sin \phi_y \sin(\theta_x - \theta_y)) \\
&& \cdot  (- \cos \phi_x \sin \phi_y + \sin \phi_x \cos \phi_y \cos (\theta_x - \theta_y)) \\
&& + v_2^x v_1^y \frac{\partial v_2^y}{\partial \theta_y}P_l'(h) (- \sin \phi_x \cos \phi_y \sin (\theta_x - \theta_y)) \end{eqnarray*}
\begin{eqnarray*}
+ \mbox{$T_1$ (block 2), S2:} && v_2^x v_1^y v_2^y P_l'''(h) (- \sin \phi_x \sin \phi_y \sin (\theta_x - \theta_y)) ( \sin \phi_x \sin \phi_y \sin (\theta_x - \theta_y)) \\
&& \cdot  (- \cos \phi_x \sin \phi_y + \sin \phi_x \cos \phi_y \cos (\theta_x - \theta_y)) \\
&& + v_2^x v_1^y v_2^y P_l''(h)(\sin \phi_x \sin \phi_y \cos (\theta_x - \theta_y)) (- \cos \phi_x \sin \phi_y + \\
&& \sin \phi_x \cos \phi_y \cos (\theta_x - \theta_y))  \\
&& +  v_2^x v_1^y v_2^y P_l''(h)(\sin \phi_x \sin \phi_y \sin (\theta_x - \theta_y)) (- \sin \phi_x \cos \phi_y \sin (\theta_x - \theta_y)) \end{eqnarray*}
\begin{eqnarray*}
+ \mbox{$T_1$ (block 2), S3:} && v_2^x v_1^y v_2^y P_l''(h) (-\sin \phi_x \sin \phi_y \sin(\theta_x - \theta_y)) (\sin \phi_x \cos \phi_y \sin (\theta_x - \theta_y)) \\
&& + v_2^x v_1^y v_2^y P_l'(h) (\sin \phi_x \cos \phi_y \cos (\theta_x - \theta_y)) \end{eqnarray*}
\begin{eqnarray*}
+ \mbox{$T_1$ (block 3), S1:} && v_2^x \frac{\partial v_1^y}{\partial \phi_y} v_2^yP_l''(h)(-\sin \phi_x \sin \phi_y \sin(\theta_x - \theta_y))(\sin \phi_x \sin \phi_y \sin (\theta_x - \theta_y)) + \\
&& v_2^x \frac{\partial v_1^y}{\partial \phi_y} v_2^yP_l'(h)(\sin \phi_x \sin \phi_y \cos(\theta_x - \theta_y)) \end{eqnarray*}
\begin{eqnarray*}
+ \mbox{$T_1$ (block 3), S2:} && v_2^x v_1^y v_2^y P_l'''(h)(-\sin \phi_x \sin \phi_y \sin(\theta_x - \theta_y)) \\
&& \cdot  (- \cos \phi_x \sin \phi_y + \sin \phi_x \cos \phi_y \cos (\theta_x - \theta_y))(\sin \phi_x \sin \phi_y \sin (\theta_x - \theta_y))   \\
&&+  v_2^x v_1^y v_2^y P_l''(h)(-\sin \phi_x \cos \phi_y \sin(\theta_x - \theta_y))(\sin \phi_x \sin \phi_y \sin(\theta_x - \theta_y)) \\
&& + v_2^x v_1^y v_2^y P_l''(h) (- \cos \phi_x \sin \phi_y + \sin \phi_x \cos \phi_y \cos (\theta_x - \theta_y)) \\
&& \cdot (\sin \phi_x \sin \phi_y \cos (\theta_x - \theta_y)) \end{eqnarray*}
\begin{eqnarray*}
+ \mbox{$T_1$ (block 3), S3:} && + v_2^x v_1^y v_2^y P_l''(h) (-\sin \phi_x \sin \phi_y \sin(\theta_x - \theta_y))(\sin \phi_x \cos \phi_y \sin (\theta_x - \theta_y))  \\
&& + v_2^x v_1^y v_2^y P_l'(h)(\sin \phi_x \cos \phi_y \cos (\theta_x - \theta_y)) \end{eqnarray*}
\begin{eqnarray*}
+ \mbox{$T_1$ (block 4), S1:} && v_2^x v_2^y \frac{\partial v_2^y}{\partial \phi_y}P_l''(h) (-\sin \phi_x \sin \phi_y \sin(\theta_x - \theta_y) (-\cos \phi_x \sin \phi_y \\
&& + \sin \phi_x \cos \phi_y \cos(\theta_x - \theta_y)) + \\
&& v_2^x v_2^y \frac{\partial v_2^y}{\partial \phi_y}P_l'(h)(-\sin \phi_x \cos \phi_y \sin(\theta_x - \theta_y)) \\
\end{eqnarray*}
\begin{eqnarray*}
\mbox{$T_1$ (block 4), S2:} &&  v_2^x(v_2^y)^2 P_l'''(h)(-\sin \phi_x \sin \phi_y \sin(\theta_x - \theta_y) ) \\
&& \cdot (- \cos \phi_x \sin \phi_y  + \sin \phi_x \cos \phi_y \cos(\theta_x - \theta_y))^2 \\
&& + 2v_2^x(v_2^y)^2 P_l''(h)(- \cos \phi_x \sin \phi_y  + \sin \phi_x \cos \phi_y \cos(\theta_x - \theta_y)) \\
&& \cdot (-\sin \phi_x \cos \phi_y \sin(\theta_x - \theta_y)) 
\end{eqnarray*}

\begin{eqnarray*}
\mbox{$T_1$ (block 4), S3:} &&  v_2^x(v_2^y)^2 P_l''(h) (-\sin \phi_x \sin \phi_y \sin(\theta_x - \theta_y)) (- \cos \phi_x \cos \phi_y  - \sin \phi_x \sin \phi_y \\
&& \cdot \cos(\theta_x - \theta_y)) \\
&& + v_2^x(v_2^y)^2 P_l'(h) (\sin \phi_x \sin \phi_y \sin(\theta_x - \theta_y)) \end{eqnarray*}
\begin{eqnarray*}
+ \mbox{$T_2$ (block 1), S1:} && (- \sin^2 \phi_x v_1^x) v_1^y \frac{\partial v_1^y}{\partial \theta_y} P_l''(h)(-\sin \phi_x \cos \phi_y + \cos \phi_x \sin \phi_y \cos (\theta_x - \theta_y)) \\
&& \cdot (\sin \phi_x \sin \phi_y \sin (\theta_x - \theta_y)) + (- \sin^2 \phi_x v_1^x) v_1^y \frac{\partial v_1^y}{\partial \theta_y} P_l'(h) (\cos \phi_x \sin \phi_y \sin (\theta_x - \theta_y)) 
\end{eqnarray*}

\begin{eqnarray*}
\mbox{$T_2$ (block 1), S2:} && (- \sin^2 \phi_x v_1^x) (v_1^y)^2 P_l'''(h) (-\sin \phi_x \cos \phi_y + \cos \phi_x \sin \phi_y \cos (\theta_x - \theta_y)) \\
&& \cdot (\sin \phi_x \sin \phi_y \sin (\theta_x - \theta_y))^2 \\
&& + 2(- \sin^2 \phi_x v_1^x) (v_1^y)^2 P_l''(h)(\sin \phi_x \sin \phi_y \sin (\theta_x - \theta_y))(\cos \phi_x \sin \phi_y \sin (\theta_x - \theta_y)) \end{eqnarray*}

\begin{eqnarray*}
\mbox{$T_2$ (block 1), S3:} && (- \sin^2 \phi_x v_1^x) (v_1^y)^2 P_l''(h) (-\sin \phi_x \cos \phi_y + \cos \phi_x \sin \phi_y \cos (\theta_x - \theta_y)) \\
&& \cdot (-\sin \phi_x \sin \phi_y \cos (\theta_x - \theta_y)) \\
&& +  (- \sin^2 \phi_x v_1^x) (v_1^y)^2 P_l'(h)(- \cos \phi_x \sin \phi_y \cos (\theta_x - \theta_y)) \end{eqnarray*}

\begin{eqnarray*}
\mbox{$T_2$ (block 2), S1:} && (- \sin^2 \phi_x v_1^x) v_1^y \frac{\partial v_2^y}{\partial \theta_y}  P_l '' (h)(-\sin \phi_x \cos \phi_y + \cos \phi_x \sin \phi_y \cos (\theta_x - \theta_y)) \\
&& \cdot (- \cos \phi_x \sin \phi_y  + \sin \phi_x \cos \phi_y \cos(\theta_x - \theta_y)) \\
&& + (- \sin^2 \phi_x v_1^x) v_1^y \frac{\partial v_2^y}{\partial \theta_y}  P_l ' (h) \cdot (\sin \phi_x \sin \phi_y  + \cos \phi_x \cos \phi_y \cos(\theta_x - \theta_y)) \end{eqnarray*}

\begin{eqnarray*}
\mbox{$T_2$ (block 2), S2:} && (- \sin^2 \phi_x v_1^x) v_1^y v_2^y P_l'''(h)  (-\sin \phi_x \cos \phi_y + \cos \phi_x \sin \phi_y \cos (\theta_x - \theta_y))   \\
&& \cdot (\sin \phi_x \sin \phi_y \sin (\theta_x - \theta_y))(- \cos \phi_x \sin \phi_y  + \sin \phi_x \cos \phi_y \cos(\theta_x - \theta_y)) \\
&& + (- \sin^2 \phi_x v_1^x) v_1^y v_2^y P_l''(h)(\cos \phi_x \sin \phi_y \sin (\theta_x - \theta_y)) \\
&& \cdot (- \cos \phi_x \sin \phi_y  + \sin \phi_x \cos \phi_y \cos(\theta_x - \theta_y)) \\
&& + (- \sin^2 \phi_x v_1^x) v_1^y v_2^y P_l''(h) (\sin \phi_x \sin \phi_y \sin (\theta_x - \theta_y)) \\
&& \cdot (\sin \phi_x \sin \phi_y  + \cos \phi_x \cos \phi_y \cos(\theta_x - \theta_y)) 
\end{eqnarray*}

\begin{eqnarray*}
\mbox{$T_2$ (block 2), S3:} && (- \sin^2 \phi_x v_1^x) v_1^y v_2^y P_l''(h) (- \sin \phi_x \cos \phi_y  + \cos \phi_x \sin \phi_y \cos(\theta_x - \theta_y)) \\
&& \cdot (\sin \phi_x \cos \phi_y \sin(\theta_x - \theta_y)) \\
&& + (- \sin^2 \phi_x v_1^x) v_1^y v_2^y P_l'(h) (\cos \phi_x \cos \phi_y \sin (\theta_x - \theta_y)) \end{eqnarray*}

\begin{eqnarray*}
\mbox{$T_2$ (block 3), S1:} && (- \sin^2 \phi_x v_1^x) \frac{\partial v_1^y}{\partial \phi_y} v_2^y P_l''(h)(- \sin \phi_x \cos \phi_y  + \cos \phi_x \sin \phi_y \cos(\theta_x - \theta_y)) \\
&& \cdot (\sin \phi_x \sin \phi_y \sin(\theta_x - \theta_y) ) \\
&& + (- \sin^2 \phi_x v_1^x) \frac{\partial v_1^y}{\partial \phi_y} v_2^y P_l'(h) (\cos \phi_x \sin \phi_y \sin(\theta_x - \theta_y)) \end{eqnarray*}

\begin{eqnarray*}
\mbox{$T_2$ (block 3), S2:} && v_1^y (- \sin^2 \phi_x v_1^x) v_2^y P_l'''(h)(- \sin \phi_x \cos \phi_y  + \cos \phi_x \sin \phi_y \cos(\theta_x - \theta_y)) \\
&& \cdot (- \cos \phi_x \sin \phi_y  + \sin \phi_x \cos \phi_y \cos(\theta_x - \theta_y)) (\sin \phi_x \sin \phi_y \sin (\theta_x - \theta_y))  \\
&& + v_1^y (- \sin^2 \phi_x v_1^x) v_2^y P_l''(h) (\sin \phi_x \sin \phi_y  + \cos \phi_x \cos \phi_y \cos(\theta_x - \theta_y)) \\ 
&& \cdot (\sin \phi_x \sin \phi_y \sin (\theta_x - \theta_y)) \\
&& + v_1^y (- \sin^2 \phi_x v_1^x) v_2^y P_l''(h) (-\cos \phi_x \sin \phi_y  + \sin \phi_x \cos \phi_y \cos(\theta_x - \theta_y)) \\ 
&& \cdot (\cos \phi_x \sin \phi_y \sin (\theta_x - \theta_y))
\end{eqnarray*}

\begin{eqnarray*}
\mbox{$T_2$ (block 3), S3:} &&  (- \sin^2 \phi_x v_1^x) v_1^y v_2^y P_l''(h) (- \sin \phi_x \cos \phi_y  + \cos \phi_x \sin \phi_y \cos(\theta_x - \theta_y)) (\sin \phi_x \cos \phi_y \\
&& \cdot \sin (\theta_x - \theta_y)) \\
&& +  (- \sin^2 \phi_x v_1^x) v_1^y v_2^y P_l'(h)(\cos \phi_x \cos \phi_y \sin (\theta_x - \theta_y)) \end{eqnarray*}

\begin{eqnarray*}
\mbox{$T_2$ (block 4), S1:} && (- \sin^2 \phi_x v_1^x) v_2^y \frac{\partial v_2^y}{\partial \phi_y} P_l''(h)  (- \sin \phi_x \cos \phi_y  + \cos \phi_x \sin \phi_y \cos(\theta_x - \theta_y)) \\ 
&& \cdot (- \cos \phi_x \sin \phi_y  + \sin \phi_x \cos \phi_y \cos(\theta_x - \theta_y)) \\
&& + (- \sin^2 \phi_x v_1^x) v_2^y \frac{\partial v_2^y}{\partial \phi_y} P_l'(h)(\sin \phi_x \sin \phi_y  + \cos \phi_x \cos \phi_y \cos(\theta_x - \theta_y)) \end{eqnarray*}

\begin{eqnarray*}
\mbox{$T_2$ (block 4), S2:} && (- \sin^2 \phi_x v_1^x) (v_2^y)^2 P_l'''(h)  (- \sin \phi_x \cos \phi_y  + \cos \phi_x \sin \phi_y \cos(\theta_x - \theta_y)) (- \cos \phi_x \sin \phi_y  \\
&& + \sin \phi_x \cos \phi_y \cos(\theta_x - \theta_y))^2 \\
&& +  2(- \sin^2 \phi_x v_1^x) (v_2^y)^2 P_l''(h) (- \cos \phi_x \sin \phi_y  + \sin \phi_x \cos \phi_y \cos(\theta_x - \theta_y))(\sin \phi_x \sin \phi_y  \\
&& + \cos \phi_x \cos \phi_y \cos(\theta_x - \theta_y)) \end{eqnarray*}

\begin{eqnarray*}
\mbox{$T_2$ (block 4), S3:} && (- \sin^2 \phi_x v_1^x) (v_2^y)^2 P_l''(h) (- \sin \phi_x \cos \phi_y  + \cos \phi_x \sin \phi_y \cos(\theta_x - \theta_y)) (- \cos \phi_x \cos \phi_y \\
&&   - \sin \phi_x \sin \phi_y \cos(\theta_x - \theta_y)) \\
&& + (- \sin^2 \phi_x v_1^x) (v_2^y)^2 P_l'(h) (\sin \phi_x \cos \phi_y   - \cos \phi_x \sin \phi_y \cos(\theta_x - \theta_y)) 
\end{eqnarray*}
\begin{eqnarray*} 
&& = \mbox{final} : \Big( v_2^x v_1^x \frac{\partial v_1^x}{\partial \theta_x} \sin^2 \phi_x + (v_2^x)^2 v_1^x \sin \phi_x \cos \phi_x + (v_2^x)^2\frac{\partial v_1^x}{\partial \phi_x} \sin^2 \phi_x \\
 && + (v_2^x)^2 v_1^x \sin \phi_x \cos \phi_x + (v_1^x)^3 \cos \phi_x \sin^3 \phi_x - (v_1^x)^2 \frac{\partial v_2^x}{\partial \theta_x} \sin^2 \phi_x - v_1^x v_2^x \frac{\partial v_2^x}{\partial \phi_x} \sin^2 \phi_x   \Big)P_l'(1)
\end{eqnarray*}

\noindent \underline{$a_{44}$}: Again, we set $T_1 = v_1^x \frac{\partial}{\partial \theta_x}$ and $T_2 = v_2^x \frac{\partial}{\partial \phi_x}$.  As this calculaton is dependent on that for $a_{24}$, we continue to use the previous labeling.  Due to the complexity of this entry, \textit{for each resulting block of terms, we only keep those which will not vanish after restricting to the diagonal}.
\begin{eqnarray*}
\mbox{$T_1$ ($T_1$ (block 1) S1)}: && v_1^x \frac{\partial v_1^x}{\partial \theta_x} v_1^y \frac{\partial v_1^y}{\partial \theta_y} P_l'(h)( \sin \phi_x \sin \phi_y \cos (\theta_x - \theta_y))
\end{eqnarray*}
\begin{eqnarray*}
\mbox{$T_1$ ($T_1$ (block 1) S2)}: && 2 (v_1^x)^2 (v_1^y)^2 P_l''(h) (\sin \phi_x \sin \phi_y \cos(\theta_x - \theta_y))^2
\end{eqnarray*}
\begin{eqnarray*}
\mbox{$T_1$ ($T_1$ (block 1) S3)}: && (v_1^x)^2 (v_1^y)^2P_l''(h)(-\sin \phi_x \sin \phi_y \cos (\theta_x - \theta_y))(- \sin \phi_x \sin \phi_y \cos (\theta_x - \theta_y)) \\
&& + (v_1^x)^2 (v_1^y)^2P_l'(h) (\sin \phi_x \sin \phi_y \cos (\theta_x - \theta_y)) 
\end{eqnarray*}
\begin{eqnarray*}
\mbox{$T_1$ ($T_1$ (block 2) S1)}: && (v_1^x)^2 v_1^y \frac{\partial v_2^y}{\partial \theta_y} P_l'(h) (- \sin \phi_x \cos \phi_y \cos (\theta_x - \theta_y))
\end{eqnarray*}
\begin{eqnarray*}
\mbox{$T_1$ ($T_1$ (block 2) S2)}: && \mbox{completely vanishes}
\end{eqnarray*}
\begin{eqnarray*}
\mbox{$T_1$ ($T_1$ (block 2) S3)}: && v_1^x \frac{\partial v_1^x}{\partial \theta_x}v_1^y v_2^y P_l'(h) (\sin \phi_x \cos \phi_y \cos (\theta_x - \theta_y))
\end{eqnarray*}
\begin{eqnarray*}
\mbox{$T_1$ ($T_1$ (block 3) S1)}: && v_1^x \frac{\partial v_1^x}{\partial \theta_x}\frac{\partial v_1^y}{\partial \phi_y} v_2^y P_l'(h) (\sin \phi_x \sin \phi_y \cos (\theta_x - \theta_y))
\end{eqnarray*}
\begin{eqnarray*}
\mbox{$T_1$ ($T_1$ (block 3) S2)}: && \mbox{completely vanishes}
\end{eqnarray*}
\begin{eqnarray*}
\mbox{$T_1$ ($T_1$ (block 3) S3)}: && v_1^x \frac{\partial v_1^x}{\partial \theta_x} v_1^y v_2^y P_l'(h) (\sin \phi_x \cos \phi_y \cos (\theta_x - \theta_y))
\end{eqnarray*}
\begin{eqnarray*}
\mbox{$T_1$ ($T_1$ (block 4) S1)}: && (v_1^x)^2 v_2^y \frac{\partial v_2^y}{\partial \phi_y}P_l'(h) (-\sin \phi_x \cos \phi_y \cos( \theta_x - \theta_y))
\end{eqnarray*}
\begin{eqnarray*}
\mbox{$T_1$ ($T_1$ (block 4) S2)}: && \mbox{completely vanishes}
\end{eqnarray*}
\begin{eqnarray*}
\mbox{$T_1$ ($T_1$ (block 4) S3)}: && (v_1^x)^2(v_2^y)^2P_l''(h) (-\sin \phi_x \sin \phi_y \cos (\theta_x - \theta_y))(-\cos \phi_x \cos \phi_y \\
&&  - \sin \phi_x \sin \phi_y \cos (\theta_x - \theta_y)) \\
&& + (v_1^x)^2 (v_2^y)^2 P_l'(h) (\sin \phi_x \sin \phi_y \cos (\theta_x - \theta_y))
\end{eqnarray*}
\begin{eqnarray*}
\mbox{$T_1$ ($T_2$ (block 1) S1)}: && v_1^x v_2^x v_1^y \frac{\partial v_1^y}{\partial \theta_y} P_l'(h) (\cos \phi_x \sin \phi_y \cos(\theta_x - \theta_y))
\end{eqnarray*}
\begin{eqnarray*}
\mbox{$T_1$ ($T_2$ (block 1) S2)}: && \mbox{completely vanishes}
\end{eqnarray*}
\begin{eqnarray*}
\mbox{$T_1$ ($T_2$ (block 1) S3)}: && v_1^x \frac{\partial v_2^x}{\partial \theta_x} (v_1^y)^2 P_l'(h)(-\cos \phi_x \sin \phi_y \cos (\theta_x - \theta_y))
\end{eqnarray*}
\begin{eqnarray*}
\mbox{$T_1$ ($T_2$ (block 2) S1)}: && v_1^x \frac{\partial v_2^x}{\partial \theta_x}v_1^y \frac{\partial v_2^y}{\partial \theta_y}P_l'(h)(\sin \phi_x \sin \phi_y + \cos \phi_x \cos \phi_y \cos (\theta_x - \theta_y))
\end{eqnarray*}
\begin{eqnarray*}
\mbox{$T_1$ ($T_2$ (block 2) S2)}: && v_1^x v_2^x v_1^y v_2^y P_l''(h)(\sin \phi_x \sin \phi_y \cos (\theta_x - \theta_y)) (\sin \phi_x \sin \phi_y + \cos \phi_x \cos \phi_y \\
&& \cdot \cos (\theta_x - \theta_y))
\end{eqnarray*}
\begin{eqnarray*}
\mbox{$T_1$ ($T_2$ (block 2) S3)}: v_1^x v_2^x v_1^y v_2^y P_l'(h) (\cos \phi_x \cos \phi_y \cos (\theta_x - \theta_y))
\end{eqnarray*}

\begin{eqnarray*}
\mbox{$T_1$ ($T_2$ (block 3) S1)}: && v_1^x v_2^x \frac{\partial v_1^y}{\partial \phi_y}v_2^y P_l'(h)(\cos \phi_x \sin \phi_y \cos(\theta_x - \theta_y))
\end{eqnarray*}
\begin{eqnarray*}
\mbox{$T_1$ ($T_2$ (block 3) S2)}: && v_1^x v_1^y v_2^x v_2^y P_l''(h)(\sin \phi_x \sin \phi_y + \cos \phi_x \cos \phi_y \cos (\theta_x - \theta_y)) \\
&& \cdot (\sin \phi_x \sin \phi_y \cos(\theta_x - \theta_y))
\end{eqnarray*}
\begin{eqnarray*}
\mbox{$T_1$ ($T_2$ (block 3) S3)}: && v_1^x v_2^x  v_1^y v_2^y P_l'(h)(\cos \phi_x \cos \phi_y \cos(\theta_x - \theta_y))
\end{eqnarray*}
\begin{eqnarray*}
\mbox{$T_1$ ($T_2$ (block 4) S1)}: && v_1^x \frac{\partial v_2^x}{\partial \theta_x} v_2^y \frac{\partial v_2^x}{\partial \phi_y}P_l'(h) (\sin \phi_x \sin \phi_y + \cos \phi_x \cos \phi_y \cos(\theta_x - \theta_y)) 
\end{eqnarray*}
\begin{eqnarray*}
\mbox{$T_1$ ($T_2$ (block 4) S2)}: && \mbox{completely vanishes}
\end{eqnarray*}
\begin{eqnarray*}
\mbox{$T_1$ ($T_2$ (block 4) S3)}: && \mbox{completely vanishes}
\end{eqnarray*}

\begin{eqnarray*}
\mbox{$T_2$ ($T_1$ (block 1) S1)}: && v_2^x \frac{\partial v_1^x}{\partial \phi_x} v_1^y \frac{\partial v_1^y}{\partial \theta_y}P_l'(h)(\sin \phi_x \sin \phi_y \cos (\theta_x - \theta_y)) \\
&& + v_2^x v_1^x v_1^y \frac{\partial v_1^y}{\partial \theta_y}P_l'(h)(\cos \phi_x \sin \phi_y \cos (\theta_x - \theta_y))
\end{eqnarray*}
\begin{eqnarray*}
\mbox{$T_2$ ($T_1$ (block 1) S2)}: && \mbox{completely vanishes}
\end{eqnarray*}
\begin{eqnarray*}
\mbox{$T_2$ ($T_1$ (block 1) S3)}: && \mbox{completely vanishes}
\end{eqnarray*}
\begin{eqnarray*}
\mbox{$T_2$ ($T_1$ (block 2) S1)}: && \mbox{completely vanishes}
\end{eqnarray*}
\begin{eqnarray*}
\mbox{$T_2$ ($T_1$ (block 2) S2)}: && v_2^x v_1^x v_1^y v_2^y P_l''(h) (\sin \phi_x \sin \phi_y \cos (\theta_x - \theta_y))(\sin \phi_x \sin \phi_y + \cos \phi_x \cos \phi_y \\
&& \cdot \cos(\theta_x - \theta_y))
\end{eqnarray*}
\begin{eqnarray*}
\mbox{$T_2$ ($T_1$ (block 2) S3)}: && v_2^x v_1^x v_1^y v_2^y P_l'(h)(\cos \phi_x \cos \phi_y \cos (\theta_x - \theta_y)) \\
&& + v_2^x \frac{\partial v_1^x}{\partial \phi_x} v_1^y v_2^y P_l'(h)(\sin \phi_x \cos \phi_y \cos (\theta_x - \theta_y))
\end{eqnarray*}
\begin{eqnarray*}
\mbox{$T_2$ ($T_1$ (block 3) S1)}: && v_2^x v_1^x \frac{\partial v_1^y}{\partial \phi_y}v_2^y P_l'(h) (\cos \phi_x \sin \phi_y \cos (\theta_x - \theta_y)) \\
 && + v_2^x \frac{\partial v_1^x}{\partial \phi_x} \frac{\partial v_1^y}{\partial \phi_y} v_2^y P_l'(h)(\sin \phi_x \sin \phi_y \cos (\theta_x - \theta_y))
\end{eqnarray*}
\begin{eqnarray*}
\mbox{$T_2$ ($T_1$ (block 3) S2)}: && v_2^x v_1^x v_1^y v_2^y P_l''(h) (\sin \phi_x \sin \phi_y + \cos\phi_x \cos \phi_y \cos(\theta_x - \theta_y)) \\
&& \cdot (\sin \phi_x \sin \phi_y \cos (\theta_x - \theta_y))
\end{eqnarray*}
\begin{eqnarray*}
\mbox{$T_2$ ($T_1$ (block 3) S3)}: && v_2^x \frac{\partial v_1^x}{\partial \phi_x} v_1^y v_2^y P_l'(h)(\sin \phi_x \cos \phi_y \cos (\theta_x - \theta_y)) \\
&& +  v_2^x v_1^x v_1^y v_2^y P_l'(h)(\cos \phi_x \cos \phi_y \cos (\theta_x - \theta_y))
\end{eqnarray*}
\begin{eqnarray*}
\mbox{$T_2$ ($T_1$ (block 4) S1)}: && \mbox{completely vanishes}
\end{eqnarray*}
\begin{eqnarray*}
\mbox{$T_2$ ($T_1$ (block 4) S2)}: && \mbox{completely vanishes}
\end{eqnarray*}
\begin{eqnarray*}
\mbox{$T_2$ ($T_1$ (block 4) S3)}: && \mbox{completely vanishes}
\end{eqnarray*}

\begin{eqnarray*}
\mbox{$T_2$ ($T_2$ (block 1) S1)}: && \mbox{completely vanishes}
\end{eqnarray*}
\begin{eqnarray*}
\mbox{$T_2$ ($T_2$ (block 1) S2)}: && \mbox{completely vanishes}
\end{eqnarray*}
\begin{eqnarray*}
\mbox{$T_2$ ($T_2$ (block 1) S3)}: && (v_2^x)^2(v_1^y)^2P_l''(h)(-\cos \phi_x \cos \phi_y - \sin \phi_x \sin \phi_y \cos(\theta_x - \theta_y)) \\
&& \cdot (-\sin \phi_x \sin \phi_y \cos(\theta_x - \theta_y)) \\
&& + v_2^x \frac{\partial v_2^x}{\partial \phi_x}(v_1^y)^2 P_l'(h) (-\cos \phi_x \sin \phi_y \cos (\theta_x - \theta_y)) \\
&& + (v_2^x)^2(v_1^y)^2P_l'(h)(\sin \phi_x \sin \phi_y \cos (\theta_x - \theta_y))
\end{eqnarray*}
\begin{eqnarray*}
\mbox{$T_2$ ($T_2$ (block 2) S1)}: && v_2^x \frac{\partial v_2^x}{\partial \phi_x} v_1^y \frac{\partial v_2^y}{\partial \theta_y} P_l'(h)(\sin \phi_x \sin \phi_y + \cos \phi_x \cos \phi_y \cos(\theta_x - \theta_y))
\end{eqnarray*}
\begin{eqnarray*}
\mbox{$T_2$ ($T_2$ (block 2) S2)}: && \mbox{completely vanishes}
\end{eqnarray*}
\begin{eqnarray*}
\mbox{$T_2$ ($T_2$ (block 2) S3)}: && \mbox{completely vanishes}
\end{eqnarray*}
\begin{eqnarray*}
\mbox{$T_2$ ($T_2$ (block 3) S1)}: && \mbox{completely vanishes}
\end{eqnarray*}
\begin{eqnarray*}
\mbox{$T_2$ ($T_2$ (block 3) S2)}: && \mbox{completely vanishes}
\end{eqnarray*}
\begin{eqnarray*}
\mbox{$T_2$ ($T_2$ (block 3) S3)}: && \mbox{completely vanishes}
\end{eqnarray*}
\begin{eqnarray*}
\mbox{$T_2$ ($T_2$ (block 4) S1)}: && v_2^x  \frac{\partial v_2^x}{\partial \phi_x} \frac{\partial v_2^y}{\partial \phi_y} v_2^y P_l'(h)(\sin \phi_x \sin \phi_y + \cos \phi_x \cos \phi_y \cos(\theta_x - \theta_y))
\end{eqnarray*}
\begin{eqnarray*}
\mbox{$T_2$ ($T_2$ (block 4) S2)}: && 2 (v_2^x)^2 (v_2^y)^2P_l''(h)(\sin \phi_x \sin \phi_y + \cos \phi_x \cos \phi_y \cos(\theta_x - \theta_y))^2
\end{eqnarray*}
\begin{eqnarray*}
\mbox{$T_2$ ($T_2$ (block 4) S3)}: && (v_2^x)^2 (v_2^y)^2P_l''(h)(-\cos \phi_x \cos \phi_y - \sin \phi_x \sin \phi_y \cos(\theta_x - \theta_y)) \\
&& \cdot (-\cos \phi_x \cos \phi_y - \sin \phi_x \sin \phi_y \cos(\theta_x - \theta_y)) \\
&& + (v_2^x)^2 (v_2^y)^2P_l'(h)(\sin \phi_x \sin \phi_y + \cos \phi_x \cos \phi_y \cos(\theta_x - \theta_y))
\end{eqnarray*}

\begin{eqnarray*}
&& = \mbox{final} : v_1^x \frac{\partial v_1^x}{\partial \theta_x } v_1^x \frac{\partial v_1^x}{\partial \theta_x} \sin^2 \phi_x P_l'(1) + 2 (v_1^x)^4 \sin^4 \phi_x P_l''(1) + (v_1^x)^4 \sin^4 \phi_x P_l''(1) \\
&& + (v_1^x)^4 \sin^2 \phi_x P_l'(1) - (v_1^x)^3 \frac{\partial v_2^x}{\partial \theta_x} \sin \phi_x \cos \phi_x P_l'(1) + (v_1^x)^2 \frac{\partial v_1^x}{\partial \theta_x} v_2^x \sin \phi_x \cos \phi_x P_l'(1) \\
&& + v_1^x \frac{\partial v_1^x}{\partial \theta_x} \frac{\partial v_1^x}{\partial \phi_x} v_2^x \sin^2 \phi_x P_l'(1) + (v_1^x)^2 \frac{\partial v_1^x}{\partial \theta_x} v_2^x \sin \phi_x \cos \phi_x P_l'(1) - (v_1^x)^2 v_2^x \frac{\partial v_2^x}{\partial \phi_x} \sin \phi_x \cos \phi_x P_l'(1)  \\
&& + (v_1^x)^2 (v_2^x)^2  \sin^2 \phi_x P_l''(1) + (v_1^x)^2 (v_2^x)^2 \sin^2 \phi_x P_l'(1) + (v_1^x)^2 \frac{\partial v_1^x}{\partial \theta_x} v_2^x \cos \phi_x \sin \phi_x P_l'(1) \\
&& - (v_1^x)^3 \frac{\partial v_2^x}{\partial \theta_x} \cos \phi_x \sin \phi_x P_l'(1) + (v_1^x)^2 \left(\frac{\partial v_2^x}{\partial \theta_x}\right)^2 P_l'(1) + (v_1^x)^2 (v_2^x)^2 \sin^2 \phi_x P_l''(1) \\
&& + (v_1^x)^2(v_2^x)^2 \cos^2 \phi_x P_l'(1) + v_1^x \frac{\partial v_1^x}{\partial \phi_x}(v_2^x)^2 \cos \phi_x \sin \phi_x P_l'(1) + (v_1^x)^2 (v_2^x)^2 \sin^2 \phi_x P_l''(1)  \\
&& + (v_1^x)^2 (v_2^x)^2 \cos^2 \phi_x  P_l'(1) + v_1^x \frac{\partial v_2^x}{\partial \theta_x} v_2^x \frac{\partial v_2^x}{\partial \phi_x} P_l'(1)  \mbox{ (from the T1 portion)} \\
&& + v_2^x \frac{\partial v_1^x}{\partial \phi_x} v_1^x \frac{\partial v_1^x}{\partial \theta_x} \sin^2 \phi_x P_l'(1) + v_2^x (v_1^x)^2 \frac{\partial v_1^x}{\partial \theta_x} \cos \phi_x \sin \phi_x P_l'(1) + (v_1^x)^2(v_2^x)^2 \sin^2 \phi_x P_l''(1) \\
&& + (v_1^x)^2 (v_2^x)^2 \cos^2 \phi_x P_l'(1) + v_1^x \frac{\partial v_1^x}{\partial \phi_x} (v_2^x)^2 \sin \phi_x \cos \phi_x P_l'(1) + v_1^x \frac{\partial v_1^x}{\partial \phi_x} (v_2^x)^2 \cos \phi_x \sin \phi_x P_l'(1) \\
&& + \left( \frac{\partial v_1^x}{\partial \phi_x} \right)^2 (v_2^x)^2 \sin^2 \phi_x P_l'(1) + (v_1^x)^2 (v_2^x)^2 \sin^2 \phi_x P_l''(1) + v_1^x \frac{\partial v_1^x}{\partial \phi_x}(v_2^x)^2 \sin \phi_x \cos \phi_x P_l'(1) \\
&& + (v_1^x)^2 (v_2^x)^2 \cos^2 \phi_x P_l'(1) + (v_2^x)^2 (v_1^x)^2 \sin^2 \phi_x P_l''(1) - (v_1^x)^2 v_2^x \frac{\partial v_2^x}{\partial \phi_x} \cos \phi_x \sin \phi_x P_l'(1) \\
&& + (v_1^x)^2 (v_2^x)^2 \sin^2 \phi_x P_l'(1)  + v_1^x v_2^x \frac{\partial v_2^x}{\partial \phi_x} \frac{\partial v_2^x}{\partial \theta_x} P_l'(1) + \left( \frac{\partial v_2^x}{\partial \phi_x} \right)^2(v_2^x)^2 P_l'(1) + 2 (v_2^x)^4 P_l''(1) \\
&& + (v_2^x)^4 P_l''(1) + (v_2^x)^4 P_l'(1) \mbox{ (from the T2 portion)}.
\end{eqnarray*}

\end{appendix}


\begin{thebibliography}{HHHHH}

\bibitem[AT00]{AT00} R. Adler, J. Taylor, \textit{Random fields and geometry}, Springer Monographs in Math., (2000).

\bibitem[AAR99]{AAR99} G. Andrews, R. Askey, and R. Roy, \textit{Special functions}, Encyclopedia of Mathematics and its Applications, Cambridge University Press, Cambridge (1999).

\bibitem[AW09]{AW09} J.M. Aza\"is, M. Wchebor, \textit{Level sets and extrema of random processes and fields}, Wiley, NJ (2009).

%\bibitem[BL13]{BL13} N. Burq, G. Lebeau, \textit{Injections de Sobolev probabilistes et applications}, Ann. Sci. ENS, Serie 4 : Tome 46 (2013) no. 6 , p. 917-962.

%\bibitem[CH15]{CH15} Y. Canzani, B. Hanin, \textit{Fixed frequency eigenfunction immersions and supremum norms of random waves}, Electronic Research Announcements in Mathematical Sciences, Volume 22, 2015, pp. 76-86.

\bibitem[CS18]{CS18} Y. Canzani, P. Sarnak, \textit{On the topology of the zero sets of monochromatic
random waves}, to appear in Comm. Pure Appl. Math. (2018).

\bibitem[DR17]{DR17} V. Dang, G. Rivi\'ere, \textit{Equidistribution of the conormal cycle of random nodal sets}, to appear in J. Euro. Math. Soc. (2018).

%\bibitem[EW18]{EW18} S. Eswarathasan, I. Wigman, \textit{Nodal domains of random spherical harmonics with a fixed number of tangencies}, preprint (2018).

\bibitem[GW14]{GW14} D. Gayet, J.Y. Welschinger, \textit{Betti numbers of random nodal sets of elliptic pseudo-differential operators}, Asian. J. Math.  (2014)

\bibitem[NS1]{NS1} F. Nazarov, M. Sodin, \textit{On the number of nodal domains of random spherical harmonics}, Amer. J. Math. 131 (2009), no. 5, 1337-1357.

\bibitem[NS10]{NS10} F. Nazarov, M. Sodin, \textit{Random Complex Zeroes and Random Nodal Lines}, Proc. ICM (2010).

\bibitem[NS2]{NS2} F. Nazarov, M. Sodin, \textit{Asymptotic laws for the spatial distribution and
the number of connected components of zero sets of Gaussian random functions},
Zh. Mat. Fiz. Anal. Geom., 12 (2016), 205-278.

%\bibitem[ORZ08]{ORZ08} F. Oravecz, Z. Rudnick, I. Wigman, \textit{The L\'eray measure of nodal sets for random eigenfunctions on the torus}, Ann. Fourier \textbf{58} 299 (2008).

\bibitem[RW18]{RW18} Z. Rudnick, I. Wigman, \textit{Points on nodal lines with given direction}, preprint, arXiv: 1802.09603 (2018).

\bibitem[SW18]{SW18} P. Sarnak, I. Wigman, \textit{Topologies of nodal sets of random band limited
functions}, to appear in Comm. Pure. Appl. Math (2018).

%\bibitem[V97]{V97} J. Vanderkam, \textit{$L^{\infty}$ norms and quantum unique ergodicity on the sphere}, Internat. Math. Res. Notices, no7, 1997, pp. 329–347.

%\bibitem[Wig09]{Wig09} I. Wigman, \textit{On the distribution of the nodal set for random spherical harmonics}, J. Math. Phys \textbf{50} (2009).

\end{thebibliography}
\end{document}